\documentclass[12pt]{article}
\usepackage{amssymb}
\usepackage{amsmath}
\usepackage{amsthm}
\usepackage{authblk}
\usepackage{color}
\usepackage{comment}
\usepackage{geometry}		
\usepackage{graphicx}
\usepackage[sort&compress]{natbib}     

\makeatletter
	\@addtoreset{equation}{section}
\makeatother

\let\originalleft\left
\let\originalright\right
\renewcommand{\left}{\mathopen{}\mathclose\bgroup\originalleft}
\renewcommand{\right}{\aftergroup\egroup\originalright}

\title{Resonant grazing bifurcations revisited.}
\author[1]{David J.W.~Simpson}
\author[2]{Indranil Ghosh}
\affil[1]{School of Mathematical and Computational Sciences, Massey University, Palmerston North, New Zealand}
\affil[2]{School of Mathematics and Statistics, University College Dublin, Dublin, D04 V1W8, Ireland}

\begin{document}

\newcommand{\cA}{\mathcal{A}}
\newcommand{\cN}{\mathcal{N}}
\newcommand{\cO}{\mathcal{O}}
\newcommand{\cP}{\mathcal{P}}
\newcommand{\cT}{\mathcal{T}}
\newcommand{\omDN}{\omega_1}
\newcommand{\omF}{\omega}
\newcommand{\rD}{{\rm D}}
\newcommand{\re}{{\rm e}}
\newcommand{\ri}{{\rm i}}
\newcommand{\myStep}[2]{{\bf Step #1} --- #2\\}

\newtheorem{theorem}{Theorem}[section]
\newtheorem{corollary}[theorem]{Corollary}
\newtheorem{lemma}[theorem]{Lemma}
\newtheorem{proposition}[theorem]{Proposition}

\theoremstyle{definition}
\newtheorem{definition}{Definition}[section]
\newtheorem{assumption}[definition]{Assumption}

\theoremstyle{remark}
\newtheorem{remark}{Remark}[section]

\maketitle




\begin{abstract} 	

In vibro-impact mechanics, the division between an impact and a near miss is a zero-velocity grazing event. Grazing bifurcations of stable periodic motions often produce complicated attractors when grazing generates a square-root term in the Poincar\'e map. This paper concerns codimension-two scenarios for which the square-root term vanishes in some iterate of the Poincar\'e map. For forced one-degree-of-freedom oscillators, this occurs when the forcing frequency is a certain rational multiple of the damped natural frequency, i.e.~the system is in resonance. In two-parameter bifurcation diagrams, curves of saddle-node and period-doubling bifurcations of single-impact periodic motions emanate from the codimension-two points. In this paper we prove these curves are quadratically tangent to the curve of grazing bifurcations, and derive explicit formulas for their quadratic coefficients. This is achieved by modifying the Poincar\'e map in a way that circumvents the square-root singularity, enabling us to use the implicit function theorem to demonstrate smoothness and perform asymptotic calculations of the saddle-node and period-doubling bifurcation curves. In doing so we resolve a long-standing conjecture on the admissibility of single-impact periodic motions by supplementing raw asymptotic computations with geometric and topological arguments. We illustrate the results with a linear impact oscillator model, matching the theoretical unfolding to numerically computed bifurcation curves. The results explain why previously reported physical experiments reveal an absence of chaos shortly past the grazing bifurcation.

\end{abstract}

\section{Introduction}
\label{sec:intro}

Many mechanical systems contain components that collide.
Examples include vibro-impact capsules \cite{LiPa20,LiWi13},
atomic force microscopes \cite{DaZh07,MiDa10},
rotors \cite{ChZh98,MoCh20},
gear assemblies \cite{HaWi07,ThNa00},
and church bells \cite{BrCh18}.
Each of these examples have interesting and physically relevant nonlinear dynamics caused by impact events.

To understand such dynamics we use a model.
If the mechanical components are rigid, it is reasonable to use a {\em hybrid model}
that combines ODEs (ordinary differential equations) for the motion between impacts,
with one or more maps that capture impact events by accounting for
velocity reversal, energy loss, impact duration, etc.~\cite{AwLa03,BlCz99,Ib09,VaSc00}.
In the phase space of the model, a map is applied when a trajectory reaches the corresponding {\em impacting surface}, Fig.~\ref{fig:Bc}.

\begin{figure}[b!]
\centering
\includegraphics[width=8cm]{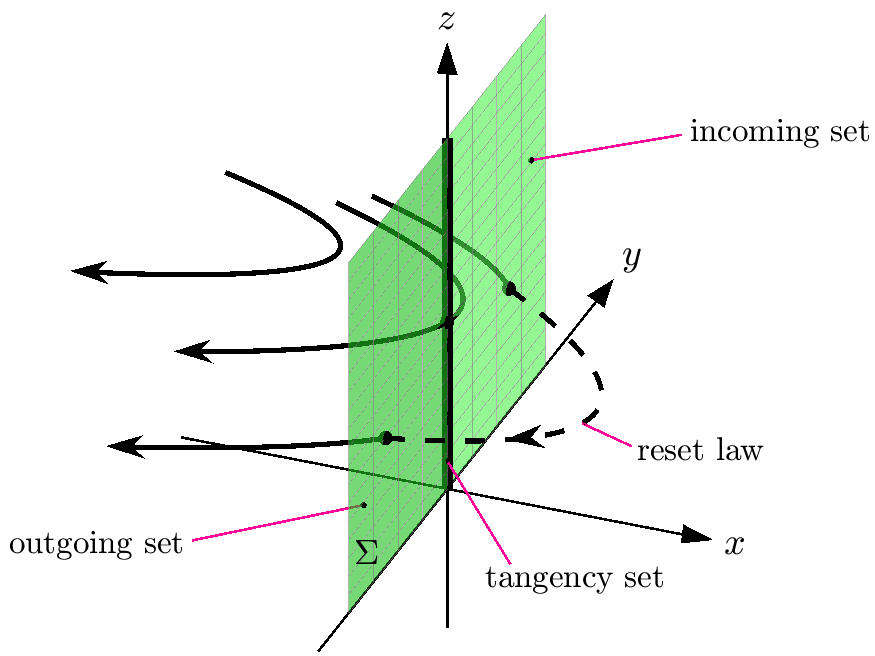}
\caption{
A sketch of phase space for the class of impacting hybrid systems
introduced in \S\ref{sec:results}.
Orbits (solid curves) obey the ODEs until reaching the impacting surface $\Sigma$ where
they are mapped under the reset law.
The impacting surface is partitioned into an incoming set, an outgoing set, and a tangency set,
as determined by the direction of the ODEs relative to the surface.
The reset law maps the incoming set to the outgoing set, and is the identity map on the tangency set.
\label{fig:Bc}}
\end{figure}

The situation that demarcates a near miss from an impact, is an impact
at which the velocity of the two objects relative to one another is zero.
In phase space, this corresponds to a trajectory arriving at an impacting surface tangentially.
When this occurs for a periodic solution, it is referred to as a {\em grazing bifurcation}.
Grazing bifurcations often only alter the paths of trajectories locally,
so cause the periodic solution to be replaced by an attractor located near this solution.
If the attractor is periodic, it must consist of some $p \ge 1$ loops near the original periodic solution,
and usually only one of these loops corresponds to an impact. 
Chin {\em et al}.~\cite{ChOt94} refer to such solutions as {\em maximal}.
Here we call them $p$-loop MPSs (maximal periodic solutions).

\begin{figure}[b!]
\centering
\includegraphics[width=15.6cm]{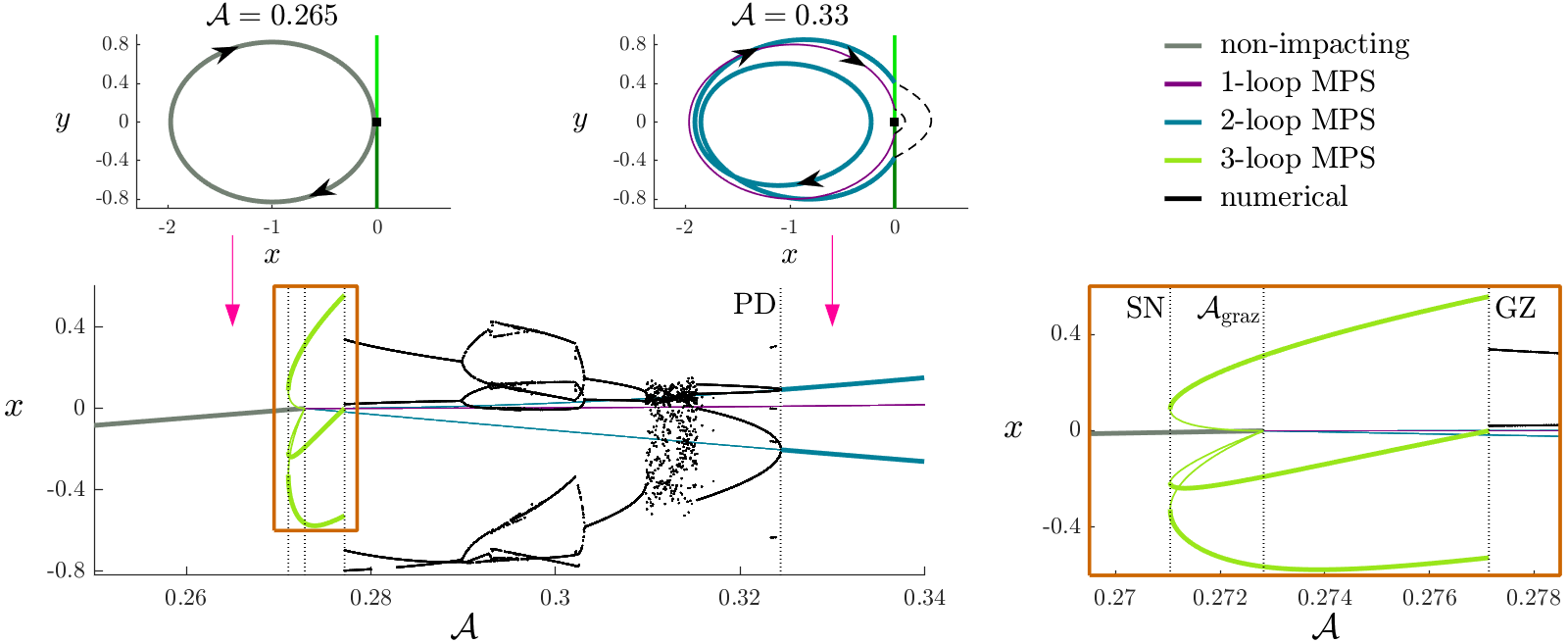}
\caption{
A bifurcation diagram of the linear impact oscillator model \eqref{eq:oscf}--\eqref{eq:oscResetLaw}
with $\zeta = 0.02$, $\epsilon = 0.9$, and $\omF = 0.854$.
The horizontal axis uses the forcing amplitude $\cA$,
while the vertical axis uses the $x$-value of $P_{\rm global}$
(so points with $x > 0$ indicate the occurrence of an impact).
The right plot is a magnification; the upper plots are phase portraits.
Branches of $p$-loop MPSs were computed by numerical continuation,
and are indicated with thick curves where the solutions are stable,
and thin curves where they are unstable.
Four bifurcations are labelled:
$\cA_{\rm graz}$: grazing bifurcation of the non-impacting periodic solution;
PD: period-doubling bifurcation of the two-loop MPS;
SN: saddle-node bifurcation of the three-loop MPS;
GZ: grazing bifurcation of the three-loop MPS.
Between GZ and PD we overlay a numerical bifurcation diagram
showing the long-term behaviour of forward orbits of random initial points.
\label{fig:Z}}
\end{figure}

Fig.~\ref{fig:Z} shows a typical example.
As the value of the parameter $\cA$ is increased,
a grazing bifurcation occurs when $\cA = \cA_{\rm graz} \approx 0.2728$.
This bifurcation generates $p$-loop MPSs for $p = 1,2,3$.
For $p = 1$ and $p = 2$, these emanate to the right of the bifurcation,
while for $p = 3$, the periodic solution emanates to the left of the bifurcation.
The three solutions are all unstable as they emanate from the grazing bifurcation,
but the two-loop MPS later gains stability in a period-doubling bifurcation,
while the three-loop MPS later gains stability in a saddle-node bifurcation.

To analyse the near-grazing dynamics, we study a Poincar\'e map.
The stability of a $p$-loop MPS is determined by the eigenvalues of the Jacobian matrix of
the $p^{\rm th}$ iterate of this map evaluated at one point of the solution.
Nordmark \cite{No91} showed that for a broad class of impacting hybrid systems,
such maps contain a square-root singularity.
Due to the square-root singularity, the Jacobian matrix contains a term
that tends to infinity at the grazing bifurcation.
Consequently the $p$-loop MPS is unstable in a neighbourhood of the bifurcation.

However, there are codimension-two scenarios at which the coefficient of the singular term vanishes,
and this is sufficient to stabilise the $p$-loop MPS as it emanates from the grazing bifurcation.
This phenomenon was first explored by Ivanov \cite{Iv93} for $p = 1$ in a linear impact oscillator model.
Ivanov argued that the one-loop MPS loses stability along curves of saddle-node and period-doubling bifurcations
that emanate quadratically from the corresponding curve of grazing bifurcations.
Foale \cite{Fo94} performed a similar analysis of the same problem
and calculated the saddle-node and period-doubling curves to leading order,
while Peterka \cite{Pe96} computed the saddle-node and period-doubling curves numerically.
Kundu {\em et al.}~\cite{KuBa11,KuBa12} rediscovered the phenomenon
and showed that bifurcation diagrams of a truncated Poincar\'e map show an absence of chaos.

Later Dankowicz and Zhao \cite{DaZh05,ZhDa06} treated the case $p=1$ in a model of a micro-actuator
and showed that leading-order theoretical expressions for the saddle-node and period-doubling curves
are consistent with numerical computations of these curves.
More recently the codimension-two scenario with $p=1$
has been observed in models of other mechanical systems \cite{MaPi09,JiCh17,YiWe20}.

Nordmark \cite{No01} and later Thota {\em et al.}~\cite{ThZh06} treated cases with $p \ge 2$.
These works presented heuristic arguments that saddle-node and period-doubling curves
emanate quadratically from the grazing curve.
For forced linear oscillators the codimension-two scenarios arise when the forcing frequency
is a certain rational multiple of the damped natural frequency \cite{Iv93,No01}.
For this reason, the scenarios are referred to as {\em resonant grazing bifurcations}.

The purpose of this paper is to unfold the $p=1$ and $p \ge 2$ cases rigorously and in a general setting.
We consider a nonlinear one-degree-of-freedom oscillator $\ddot{x} = F(x,\dot{x},t)$,
where $F$ is periodic in $t$, impacts are modelled by a reset law $R$,
and dots denote differentiation with respect to time $t$.
In order to prove that saddle-node and period-doubling curves exist and are smooth,
we work with a modification of the Poincar\'e map referred to as the VIVID function \cite{GhSi25b}.
This function uses $\dot{x}$ as an input, instead of $x$,
and consequently its derivative is non-singular.

As a first step to obtaining the results,
we determine the side of the grazing bifurcation on which the $p$-loop MPS is created,
as this is different on different sides of the codimension-two point.
Conditions for this were derived by Nordmark \cite{No01},
however, it was never proved that these conditions are sufficient
because it is challenging to show that the $p$-loop MPS is {\em admissible}.
By admissible, we mean that all of the computed points at which the $p$-loop MPS intersects the Poincar\'e section
lie on the correct side of the switching manifold dividing impacts from near misses.
Admissibility is necessary for MPSs to be valid solutions of the hybrid system.
Below we show that Nordmark's conditions are sufficient by
combining asymptotic calculations with geometric and topological arguments.

The remainder of this paper is organised as follows.
The main results are presented in \S\ref{sec:results}.
Theorems \ref{th:codim11} and \ref{th:codim1p} treat generic grazing bifurcations
and explain which side of the grazing bifurcation the $p$-loop MPSs are admissible.
Theorems \ref{th:codim21} and \ref{th:codim2p} treat resonant grazing bifurcations
and provide leading-order expressions for the saddle-node and period-doubling bifurcation curves.

In \S\ref{sec:example} we illustrate the results with a prototypical model of linear impact oscillator
using parameter values based on laboratory experiments reported by Pavlovskaia {\em et al}.~\cite{PaIn10}. 
Their experiments showed a stable two-loop MPS very shortly after a grazing bifurcation,
and we can now conclude that this occurs because the bifurcation is close to $p=2$ resonance.
For resonant grazing bifurcations with $p=1$, $2$, and $3$,
we numerically continue saddle-node and period-doubling bifurcations,
and compare these to the theoretical unfolding.

The four theorems are proved in Sections \ref{sec:generic} and \ref{sec:resonant}.
Conclusions are presented in \S\ref{sec:conc}.

\section{Main results}
\label{sec:results}

\subsection{A general class of impacting hybrid systems}
\label{sub:generalClass}

Let $x(t) \in \mathbb{R}$ denote the position
of an object subject to a collection of forces.
We write
\begin{equation}
\ddot{x} = F(x,\dot{x},t;\mu,\eta),
\label{eq:F}
\end{equation}
where $F$ is periodic in $t$ with period $\frac{2 \pi}{\omF(\mu,\eta)}$,
and $\mu,\eta \in \mathbb{R}$ are parameters.
To write \eqref{eq:F} as an autonomous system,
we let $y(t) = \dot{x}(t)$ denote the velocity of the object, and
\begin{equation}
z = \omF(\mu,\eta) t ~{\rm mod}~ 2 \pi,
\label{eq:z}
\end{equation}
denote the phase of $F$, which takes values in $[0, 2 \pi)$.
Then \eqref{eq:F} can be written as
\begin{equation}
\begin{bmatrix} \dot{x} \\ \dot{y} \\ \dot{z} \end{bmatrix}
= \begin{bmatrix}
y \\ F \big( x, y, \tfrac{z}{\omF(\mu,\eta)}; \mu, \eta \big) \\ \omF(\mu,\eta) \end{bmatrix}.
\label{eq:f}
\end{equation}

Now suppose a wall is located at $x=0$, and the object only obeys \eqref{eq:f} for $x \le 0$.
The object may hit the wall; in $(x,y,z)$-phase space
this occurs when an orbit reaches the {\em impacting surface}
\begin{equation}
\Sigma = \left\{ (0,y,z) \,\big|\, y \in \mathbb{R}, z \in [0, 2 \pi) \right\}.
\nonumber
\end{equation}
Since $\dot{x} = y$, impacting events occur at points on $\Sigma$ with $y > 0$,
or with $y = 0$ in the special case that the impact velocity is zero.
Consequently, we refer to the part of $\Sigma$ with $y > 0$ as is the {\em incoming set},
the part of $\Sigma$ with $y = 0$ as is the {\em tangency set},
and the part of $\Sigma$ with $y < 0$ as is the {\em outgoing set}, Fig.~\ref{fig:Bc}.
Since \eqref{eq:f} applies for $x \le 0$, grazing orbits have $\ddot{x} < 0$.
For this reason we refer to the subset of the tangency set
\begin{equation}
\Sigma_{\rm graz} = \left\{ (0,0,z) \,\middle|\, z \in [0, 2 \pi),\,
F \big( 0, 0, \tfrac{z}{\omF(\mu,\eta)}; \mu, \eta \big) < 0 \right\},
\nonumber
\end{equation}
as the {\em grazing set}.

With the viewpoint that the object and wall are rigid,
we model impact events with a reset law $R$ that updates the values of $y$ and $z$.
The only assumption we place upon $R$, besides smoothness,
is that it converts a positive impact velocity to a negative recoil velocity,
and leaves $y$ and $z$ unchanged when the impact velocity is zero.
That is, $R$ maps the incoming set to the outgoing set, and is the identity map on the tangency set.
With this assumption, the reset law can be written as
\begin{equation}
R(y,z;\mu,\eta)
= \begin{bmatrix} -y \Phi(y,z;\mu,\eta) \\ (z + y \Psi(y,z;\mu,\eta)) ~{\rm mod}~ 2 \pi \end{bmatrix},
\label{eq:R}
\end{equation}
where $\Phi$ and $\Psi$ are real-valued functions,
and $\Phi$ takes only positive values.

In summary, the motion of the object is modelled by \eqref{eq:f} for $x \le 0$, and
\begin{equation}
\begin{bmatrix} y \\ z \end{bmatrix}
\mapsto R(y,z;\mu,\eta),
\qquad \text{when $x = 0$ with $y > 0$},
\label{eq:resetLaw}
\end{equation}
where $R$ has the form \eqref{eq:R}.
Several research groups have formulated models of this form
and found good agreement to the physically observed behaviour of simple impacting machines
\cite{DeBi96,PiVi04,QiFe00,WiVi14}.
Often $\Phi$ is treated as a constant, in which case $\Phi$ is the coefficient of restitution.
If $\Psi = 0$, then impacts are assumed to occur instantaneously.

In the above formulation, we have incorporated the parameters $\mu$ and $\eta$ in a general manner
so that they can be set equal to the forcing frequency,
the coefficient of restitution, some other system parameter,
or any shifted version of these so that the grazing bifurcation can be assumed to occur at $\mu = 0$,
as in the theorems below.

\subsection{Grazing and a global return map}
\label{sub:grazing}

We first describe typical grazing bifurcations,
and to this end suppress the parameter $\eta$.
For a grazing bifurcation at $\mu = 0$,
we make the following assumption.

\begin{assumption}
Suppose \eqref{eq:f} with $\mu = 0$ has a periodic solution $\Gamma$
intersecting $\Sigma$ at exactly one point $(0,0,z_{\rm graz}) \in \Sigma_{\rm graz}$.
\label{ass:grazing}
\end{assumption}

\begin{figure}[b!]
\centering
\includegraphics[width=12cm]{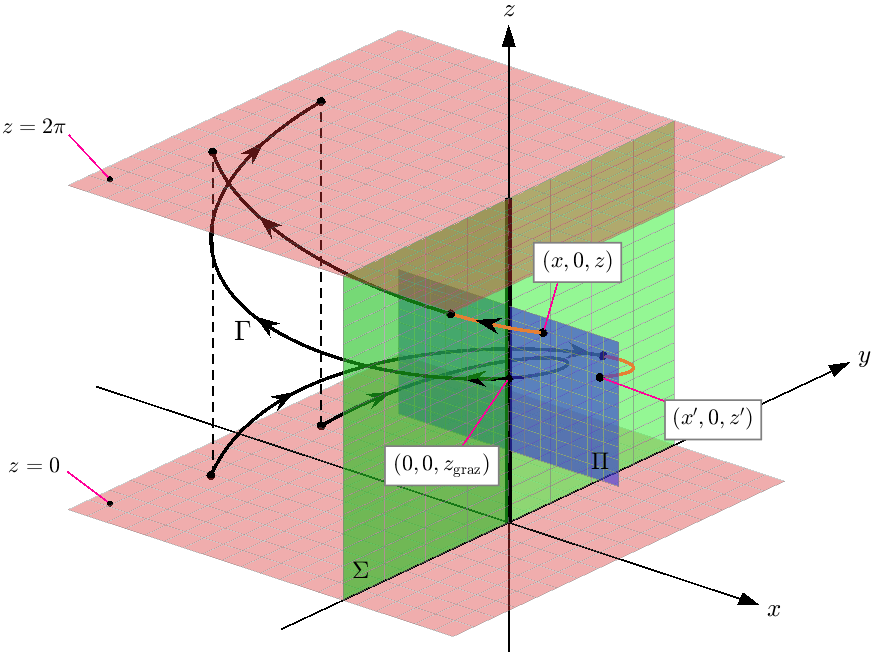}
\caption{
A sketch illustrating the $P_{\rm global}$: the return map to $\Pi$
induced by the flow of \eqref{eq:f}.
We also sketch the grazing periodic solution $\Gamma$ of Assumption \ref{ass:grazing}.
\label{fig:Ba}}
\end{figure}

In order to state quantitative theorems, we consider the
return map $(x',z') = P_{\rm global}(x,z;\mu)$, defined as follows.
Let $\Pi \subset \mathbb{R}^3$ be a compact part of the plane $y=0$
in a neighbourhood of $(0,0,z_{\rm graz})$
that intersects $\Gamma$ only at $(0,0,z_{\rm graz})$.
For any $(x,0,z) \in \Pi$, let $(x',0,z')$ be the next point
at which the forward orbit of $(x,0,z)$ under \eqref{eq:f} intersects $\Pi$, see Fig.~\ref{fig:Ba}.
In a sufficiently small neighbourhood of $(x,z;\mu) = (0,z_{\rm graz};0)$,
the map $P_{\rm global}$ is smooth because impact events are ignored
and orbits intersect $\Pi$ transversally with a return time close to the period of $\Gamma$.
Thus we can write
\begin{equation}
P_{\rm global}(x,z;\mu) = \begin{bmatrix} 0 \\ z_{\rm graz} \end{bmatrix}
+ A \begin{bmatrix} x \\ z - z_{\rm graz} \end{bmatrix} + b \mu
+ \cO \left( \left( |x| + |z-z_{\rm graz}| + |\mu| \right)^2 \right),
\label{eq:Pglobal}
\end{equation}
where
\begin{align}
A &= \rD P_{\rm global}(0,z_{\rm graz};0) = \begin{bmatrix}
\frac{\partial x'}{\partial x} & \frac{\partial x'}{\partial z} \\[1mm]
\frac{\partial z'}{\partial x} & \frac{\partial z'}{\partial z} \end{bmatrix}, \label{eq:A} \\
b &= \frac{\partial P_{\rm global}}{\partial \mu}(0,z_{\rm graz};0). \label{eq:b}
\end{align}

\subsection{Existence and stability of the non-impacting periodic solution}
\label{sub:nonimpacting}

Let
\begin{align}
\tau &= {\rm trace}(A), &
\delta &= \det(A).
\end{align}
If $\delta - \tau + 1 \ne 0$, then the matrix $I-A$ is non-singular
and in a neighbourhood of $(x,z;\mu) = (0,z_{\rm graz};0)$
the map $P_{\rm global}$ has the unique fixed point
\begin{equation}
\begin{bmatrix} x^*(\mu) \\ z^*(\mu) \end{bmatrix} = \begin{bmatrix} 0 \\ z_{\rm graz} \end{bmatrix}
+ (I - A)^{-1} b \mu + \cO \left( \mu^2 \right).
\label{eq:xStarzStar}
\end{equation}
With $\mu = 0$, this point corresponds to the grazing orbit $\Gamma$.
If $x^*(\mu) < 0$, the fixed point corresponds a non-impacting periodic solution
of \eqref{eq:f}--\eqref{eq:resetLaw}.

For each $i,j = 1,2$, we let $a_{ij}$ denote the $(i,j)$-entry of $A$,
and $b_i$ denote the $i^{\rm th}$ component of $b$.
From \eqref{eq:xStarzStar},
\begin{equation}
\frac{d x^*}{d \mu}(0) = \frac{\beta}{\delta - \tau + 1},
\label{eq:xStarzStarDeriv}
\end{equation}
where
\begin{equation}
\beta = (1 - a_{22}) b_1 + a_{12} b_2 \,.
\label{eq:beta}
\end{equation}
Thus $\beta \ne 0$ is the {\em transversality condition} for the grazing bifurcation.
This ensures the non-impacting periodic solution collides with $\Sigma$ in a generic fashion
as $\mu$ is varied through $0$.

Let
\begin{align}
\lambda_1 &= \frac{1}{2} \left( \tau + \sqrt{\tau^2 - 4 \delta} \right), &
\lambda_2 &= \frac{1}{2} \left( \tau - \sqrt{\tau^2 - 4 \delta} \right),
\label{eq:eigs}
\end{align}
denote the eigenvalues of $A$.
These are the stability multipliers (or non-trivial Floquet multipliers)
associated with $\Gamma$ when impact events are ignored,
and are either real or complex conjugates of one another.
Below we assume $|\tau| - 1 < \delta < 1$, so that $\Gamma$ is asymptotically stable,
in which case the non-impacting periodic solution is asymptotically stable for sufficiently small values of $\mu$.
Further, we assume $P_{\rm global}$ is orientation-preserving, so $\delta > 0$, as is usually the case in applications.
Hence we restrict our attention to pairs $(\tau,\delta)$ belonging to the trapesium
\begin{equation}
\cT = \left\{ (\tau,\delta) \in \mathbb{R}^2 \,\middle|\,
0 < \delta < 1, \, |\tau| < \delta+1 \right\},
\label{eq:cT}
\end{equation}
shown in Fig.~\ref{fig:D}.

\begin{figure}[b!]
\centering
\includegraphics[width=12cm]{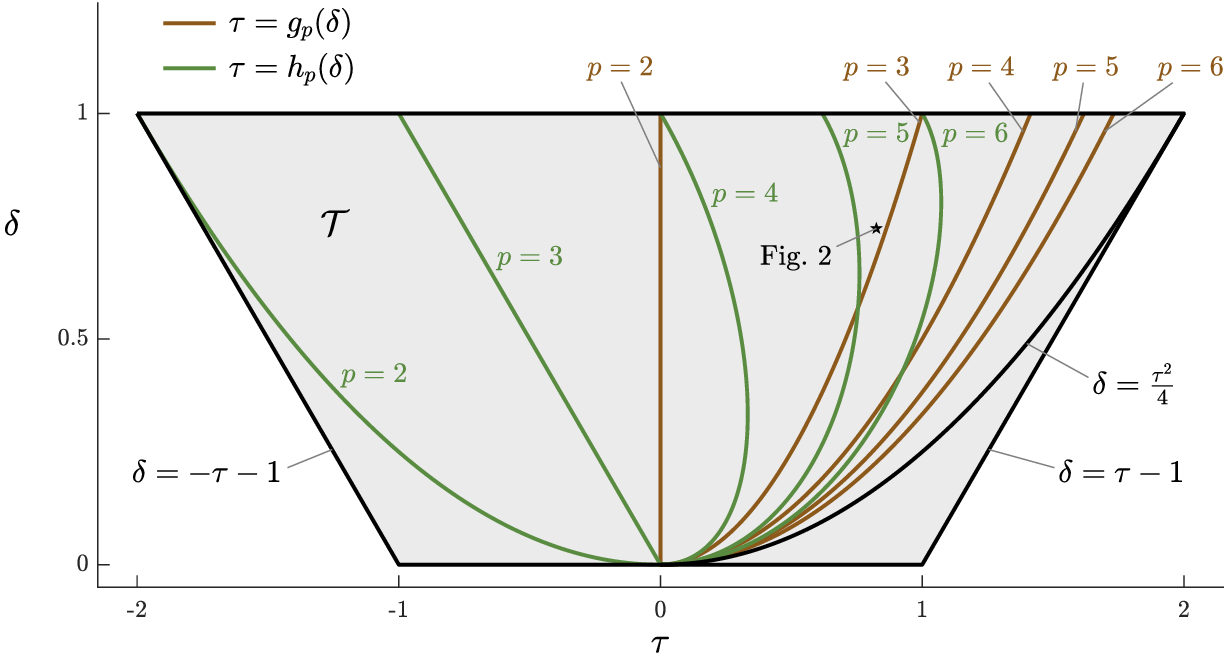}
\caption{
The curves $\tau = g_p(\delta)$ and $\tau = h_p(\delta)$ for $p = 2,3,\ldots,6$
for values within the trapesium $\cT$ \eqref{eq:cT}.
\label{fig:D}}
\end{figure}

\subsection{Maximal periodic solutions}
\label{sub:mps}

The grazing bifurcation at $\mu = 0$ can give birth to a wide variety of invariant sets.
Of these, we focus on $p$-loop MPSs (maximal periodic solutions)
which consist of $p$ loops near $\Gamma$ and have exactly one impacting event.
Theorems \ref{th:codim11} and \ref{th:codim1p} explain which $p$-loop MPSs arise,
and on which side of the grazing bifurcation they emerge.

The case $p = 1$ is straight-forward:
in generic situations a one-loop MPS is always created
and the side of the bifurcation on which it emerges is governed by the sign of $a_{12}$.
To handle $p \ge 2$, we introduce some auxiliary functions.
We first define
\begin{align}
g_p(\delta) &= 2 \sqrt{\delta} \cos \left( \frac{\pi}{p} \right), \label{eq:gp} \\
H_p(\tau,\delta) &= \sum_{j=1}^{p-1} \sum_{k=1}^j \lambda_1^{k-j} \lambda_2^{1-k}, \label{eq:Hp}
\end{align}
where $\lambda_1$ and $\lambda_2$ are given by \eqref{eq:eigs},
and in \eqref{eq:gp} we allow non-integer values of $p$.
We then define $h_2(\delta) = -2 \sqrt{\delta}$,
and for all $p \ge 3$ let $h_p(\delta)$ be the largest value of $\tau \in \mathbb{R}$ at which $H_p(\tau,\delta) = 0$.
The following bounds on $h_p(\delta)$ are established in Appendix \ref{app:someProofs}.

\begin{lemma}
For any $0 < \delta < 1$ and $p \ge 3$,
we have $g_{\frac{p}{2}}(\delta) < h_p(\delta) < g_{p-1}(\delta)$.
\label{le:admissibilityBoundary}
\end{lemma}

The curves $\tau = g_p(\delta)$ and $\tau = h_p(\delta)$ are shown in Fig.~\ref{fig:D} for $p = 2,3,\ldots,6$.
As $p \to \infty$, the curves limit to $\tau = 2 \sqrt{\delta}$.
Nordmark \cite{No01} conjectured that
if $g_p(\delta) < \tau < \delta + 1$, then
a $p$-loop MPS emanates on one side of the bifurcation,
while if $h_p(\delta) < \tau < g_p(\delta)$,
then a $p$-loop MPS emanates on the other side of the bifurcation.
Theorem \ref{th:codim1p} shows this to be true.

For example, the grazing bifurcation in Fig.~\ref{fig:Z} has $(\tau,\delta) \approx (0.8248,0.7451)$,
indicated by a star in Fig.~\ref{fig:D}.
This point lies to the right of $\tau = g_2(\delta)$, thus a two-loop MPS emanates to the right of the grazing bifurcation.
Also, $h_p(\delta) < \tau < g_p(\delta)$ for $p = 3,4,5$,
thus a $p$-loop MPS emanates to the left of the grazing bifurcation for $p = 3,4,5$
(for clarity only the $p=3$ solution is shown in Fig.~\ref{fig:Z}).
Theorem \ref{th:codim1p} does not handle values of $p$ for which $\tau < h_p(\delta)$,
indeed $p$-loop MPSs with $p = 5$ and $p \ge 7$ can arise for $(\tau,\delta) \in \cT$ when $\delta \approx 1$ \cite{No01}.
These solutions do not relate to resonant grazing, so we do not consider them in this paper.

\subsection{Unfolding generic grazing bifurcations}
\label{sub:unfoldGeneric}

To account for the impact events neglected by $P_{\rm global}$, we use a discontinuity map $P_{\rm disc}$.
As shown in \S\ref{sub:discMap},
$P_{\rm disc}$ has a $\sqrt{x}$-term
whose coefficient contains the factor
\begin{equation}
\alpha = 1 + \phi - \frac{\gamma \psi}{\omF(0)},
\label{eq:alpha}
\end{equation}
where
\begin{align}
\phi &= \Phi(0,z_{\rm graz};0), &
\psi &= \Psi(0,z_{\rm graz};0), &
\gamma &= -F \big( 0, 0, \tfrac{z_{\rm graz}}{\omF(0)}; 0).
\label{eq:phipsigamma}
\end{align}
The sign of $\alpha$ is important to the dynamics created in the grazing bifurcation.
Notice $\gamma > 0$, by the assumption that the grazing point belongs to $\Sigma_{\rm graz}$.

\begin{theorem}[generic grazing, $p=1$]
Consider an impacting hybrid system \eqref{eq:f}--\eqref{eq:resetLaw} where $F$ and $R$ are $C^3$,
and $\eta \in \mathbb{R}$ is fixed.
Suppose Assumption \ref{ass:grazing} holds, $(\tau,\delta) \in \cT$, and $\alpha > 0$.
A unique unstable one-loop MPS emanates from the grazing bifurcation
for $\mu > 0$ if $a_{12} \beta > 0$,
and for $\mu < 0$ if $a_{12} \beta < 0$.
\label{th:codim11}
\end{theorem}

Now let $\kappa_p = \tau - g_p(\delta)$.

\begin{theorem}[generic grazing, $p \ge 2$]
Consider an impacting hybrid system \eqref{eq:f}--\eqref{eq:resetLaw}
where $F$ and $R$ are $C^3$ and $\eta \in \mathbb{R}$ is fixed, and let $p \ge 2$.
Suppose Assumption \ref{ass:grazing} holds, $(\tau,\delta) \in \cT$, $a_{12} \alpha > 0$, and $\tau > h_p(\delta)$.
A unique unstable $p$-loop MPS emanates from the grazing bifurcation
for $\mu > 0$ if $\beta \kappa_p > 0$,
and for $\mu < 0$ if $\beta \kappa_p < 0$.
\label{th:codim1p}
\end{theorem}

\subsection{Unfolding resonant grazing bifurcations}
\label{sub:unfoldResonant}

Theorem \ref{th:codim11} shows that for $p = 1$, a codimension-two scenario arises when $a_{12} = 0$,
while Theorem \ref{th:codim1p} shows that for $p \ge 2$, a codimension-two scenario arises when $\tau = g_p(\delta)$.
To unfold these, we use the second parameter $\eta$
and the following more general assumption that grazing occurs at $\mu = 0$ for all values of $\eta$ in a neighbourhood
$\cN \subset \mathbb{R}$ of $0$.

\begin{assumption}
Suppose that for all $\eta \in \cN$, the system \eqref{eq:f} with $\mu = 0$
has a periodic orbit $\Gamma_\eta$ that varies continuously with $\eta$ and
intersects $\Sigma$ at exactly one point, $(0,0,z_{{\rm graz},\eta}) \in \Sigma_{\rm graz}$.
\label{ass:grazing2}
\end{assumption}

We assume that the codimension-two scenarios occur at $\eta = 0$,
and that the above quantities, $\tau$, etc, are now evaluated at $\eta = 0$.
For all $p \ge 1$, let
\begin{equation}
\xi_p = \frac{\partial^2 P_{{\rm global},1}^p}{\partial z^2} \bigg|_{(x,z;\mu,\eta) = (0,z_{{\rm graz},0};0,0)},
\label{eq:xi}
\end{equation}
where $P_{{\rm global},1}^p$ is the first component of the $p^{\rm th}$ iterate of $P_{\rm global}$.
For $p = 1$, let
\begin{equation}
s^\pm_1 = (1 \mp a_{22}) \left( a_{11} \phi^2 \mp 1 \right) + \frac{\alpha^2 \omF^2 \xi_1}{(1 - a_{22}) \gamma}.
\label{eq:c1}
\end{equation}
Also let $a_{12}' = \frac{d a_{12}}{d \eta} \big|_{\eta = 0}$, and
\begin{align}
c_{{\rm SN},1} &= \frac{\left( \alpha \omF a_{12}' \right)^2}{2 \beta \gamma s^+_1}, &
c_{{\rm PD},1} &= \frac{s^+_1}{s^-_1} \left( 2 - \frac{s^+_1}{s^-_1} \right) c_{{\rm SN},1} \,.
\label{eq:coeffs1}
\end{align}

\begin{theorem}[resonant grazing, $p=1$]
Consider an impacting hybrid system \eqref{eq:f}--\eqref{eq:resetLaw} where $F$ and $R$ are $C^k$ ($k \ge 5$).
Suppose Assumption \ref{ass:grazing2} holds, $(\tau,\delta) \in \cT$,
$\alpha > 0$, $\beta \ne 0$, $a_{12} = 0$, $a_{12}' \ne 0$, and $s^\pm_1 \ne 0$. 
Then there exist local $C^{k-2}$ functions
\begin{align}
g_{{\rm SN},1}(\eta) &= c_{{\rm SN},1} \eta^2
+ \cO \left( \eta^3 \right), \label{eq:sn1} \\
g_{{\rm PD},1}(\eta) &= c_{{\rm PD},1} \eta^2
+ \cO \left( \eta^3 \right), \label{eq:pd1}
\end{align}
such that a one-loop MPS undergoes
a saddle-node bifurcation when $\mu = g_{{\rm SN},1}(\eta)$,
with ${\rm sgn}(\eta) = {\rm sgn} \left( a_{12}' s^+_1 \right)$,
and has a stability multiplier of $-1$ when $\mu = g_{{\rm PD},1}(\eta)$,
with ${\rm sgn}(\eta) = {\rm sgn} \left( a_{12}' s^-_1 \right)$.
\label{th:codim21}
\end{theorem}

For all $p \ge 2$, let
\begin{equation}
s^\pm_p = \left( 1 \pm \delta^{\frac{p}{2}} \right) \left( -\delta^{\frac{p}{2}} \phi^2 \mp 1 \right)
+ \frac{\alpha^2 \omF^2 \xi_p}{\left( 1 + \delta^{\frac{p}{2}} \right) \gamma}.
\label{eq:cp}
\end{equation}
Also let $\kappa_p' = \frac{d \kappa_p}{d \eta} \big|_{\eta = 0}$, and
\begin{align}
c_{{\rm SN},p} &= \left( \frac{p a_{12} \alpha \omF \delta^{\frac{p}{2} - 1} \kappa_p'}
{2 \sin^2 \!\big( \frac{\pi}{p} \big) \left( 1 + \delta^{\frac{p}{2}} \right)} \right)^2 \frac{\delta - \tau + 1}{2 \beta \gamma s^+_p}, &
c_{{\rm PD},p} &= \frac{s^+_p}{s^-_p} \left( 2 - \frac{s^+_p}{s^-_p} \right) c_{{\rm SN},p} \,.
\label{eq:coeffsp}
\end{align}

\begin{theorem}[resonant grazing, $p \ge 2$]
Consider an impacting hybrid system \eqref{eq:f}--\eqref{eq:resetLaw}
where $F$ and $R$ are $C^k$ ($k \ge 5$), and let $p \ge 2$.
Suppose Assumption \ref{ass:grazing2} holds, $(\tau,\delta) \in \cT$, 
$a_{12} \alpha > 0$, $\beta \ne 0$, $\kappa_p = 0$, $\kappa_p' \ne 0$, and $s^\pm_p \ne 0$. 
Then there exist local $C^{k-2}$ functions
\begin{align}
g_{{\rm SN},p}(\eta) &= c_{{\rm SN},p} \eta^2
+ \cO \left( \eta^3 \right), \label{eq:snp} \\
g_{{\rm PD},p}(\eta) &= c_{{\rm PD},p} \eta^2
+ \cO \left( \eta^3 \right), \label{eq:pdp}
\end{align}
such that a $p$-loop MPS undergoes
a saddle-node bifurcation when $\mu = g_{{\rm SN},p}(\eta)$, with ${\rm sgn}(\eta) = {\rm sgn} \left( \kappa_p' s^+_p \right)$,
and has a stability multiplier of $-1$ when $\mu = g_{{\rm PD},p}(\eta)$, with ${\rm sgn}(\eta) = {\rm sgn} \left( \kappa_p' s^-_p \right)$.
\label{th:codim2p}
\end{theorem}

\begin{figure}[b!]
\centering
\includegraphics[width=6.4cm]{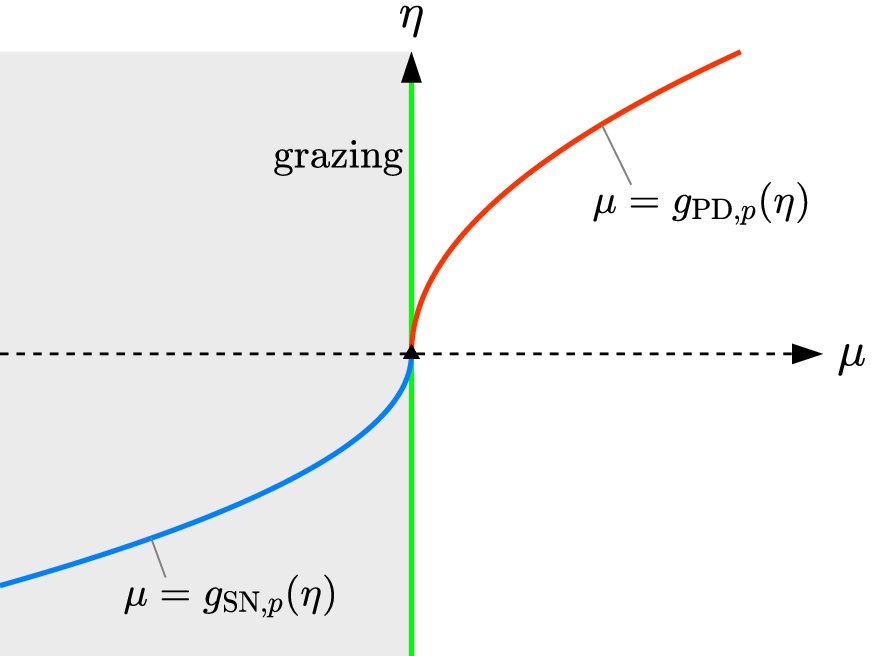}
\caption{
A sketch of the bifurcation curves described
by Theorems \ref{th:codim21} and \ref{th:codim2p}
in the case $c_{{\rm SN},p} < 0$, $c_{{\rm PD},p} > 0$, and $\beta > 0$.
Here a stable non-impacting periodic solution exists for $\mu < 0$.
\label{fig:E}}
\end{figure}

Fig.~\ref{fig:E} provides a sketch of a typical two-parameter bifurcation diagram
showing the curves predicted by Theorems \ref{th:codim21} and \ref{th:codim2p}.
In generic situations, $\mu = g_{{\rm PD},p}(\eta)$ is a curve of period-doubling bifurcations of the $p$-loop MPS.
We have not computed the quadratic and cubic terms necessary to evaluate the non-degeneracy condition
that ensures period-doubling bifurcations do occur on this curve,
as this appears to be extremely challenging to achieve in a general setting.
For the curve $\mu = g_{{\rm SN},p}(\eta)$, we compute in \S\ref{sec:resonant} the key quadratic term
that shows saddle-node bifurcation occur on this curve.

\begin{remark}
In Theorem \ref{th:codim21}, $a_{12} = 0$, thus $\beta = (1 - a_{22}) b_1$.
Hence in this case the assumption $\beta \ne 0$ implies $a_{22} \ne 1$ and $b_1 \ne 0$.
\end{remark}

\section{A harmonically forced linear oscillator}
\label{sec:example}

In this section we apply the above theory
to the forced linear oscillator depicted in Fig.~\ref{fig:F}.

\begin{figure}[b!]
\centering
\includegraphics[width=9cm]{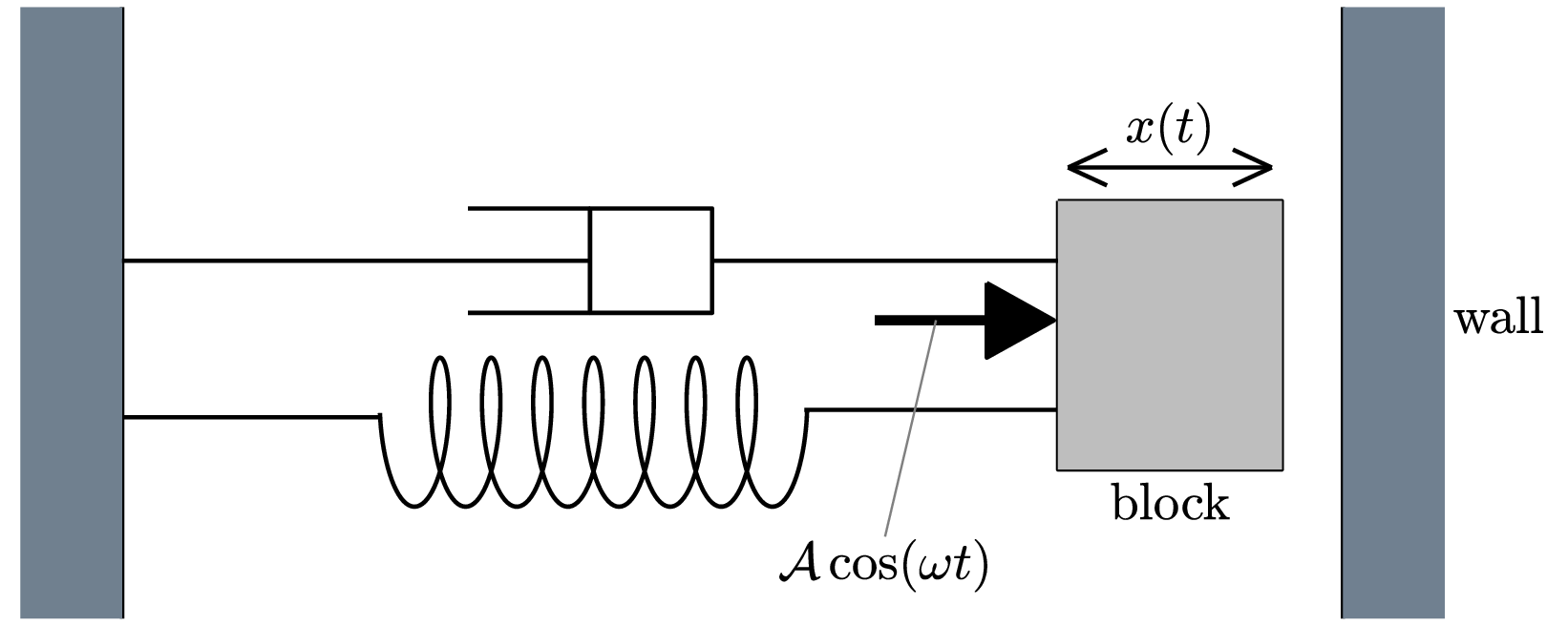}
\caption{
A sketch of the impact oscillator modelled by \eqref{eq:oscf}--\eqref{eq:oscResetLaw}.
\label{fig:F}}
\end{figure}

\subsection{Equations of motion and grazing of the non-impacting periodic solution}
\label{sub:1}

For the motion of the block in Fig.~\ref{fig:F}, we use the non-dimensionalised equation
\begin{equation}
\ddot{x} + 2 \zeta \dot{x} + x + 1 = \cA \cos(\omF t),
\label{eq:oscf}
\end{equation}
where $\zeta$ is the damping ratio,
$\cA$ is the forcing amplitude,
and $\omF$ is the forcing frequency.
We assume impacts of the block with the wall occur instantaneously
with coefficient of restitution $\epsilon$:
\begin{equation}
y \mapsto -\epsilon y, \qquad \text{when $x = 0$ with $y > 0$}.
\label{eq:oscResetLaw}
\end{equation}

The equilibrium position of the oscillator is $x = -1$.
With $0 < \zeta < 1$, the oscillator is under-damped
and the eigenvalues associated with the equilibrium are $-\zeta \pm \ri \omDN$, where
\begin{equation}
\omDN = \sqrt{1 - \zeta^2}
\nonumber
\end{equation}
is the damped natural frequency.
Given $(x,\dot{x}) = (x_0,y_0)$ at time $t_0$,
the solution to \eqref{eq:oscf} is
\begin{align}
\phi(t;x_0,y_0,t_0;\cA) &= \re^{-\zeta(t-t_0)} \Big( \big( \cos(\omDN(t-t_0))
+ \tfrac{\zeta}{\omDN} \,\sin(\omDN(t-t_0)) \big) \big( x_0 - \phi_p(t_0;\cA) \big) \nonumber \\
&\quad+\,\tfrac{1}{\omDN} \,\sin(\omDN(t-t_0)) \big( y_0 - \dot{\phi}_p(t_0;\cA) \big) \Big) + \phi_p(t;\cA),
\label{eq:flow}
\end{align}
where
\begin{equation}
\phi_p(t;\cA) = -1 + \frac{\cA}{\left( 1-\omF^2 \right)^2 + 4 \zeta^2 \omF^2}
\left( \left( 1-\omF^2 \right) \cos(\omF t) + 2 \zeta \omF \sin(\omF t) \right).
\label{eq:phip}
\end{equation}
As $t \to \infty$, the solution converges to $\phi_p(t)$,
which is an asymptotically stable period-$\frac{2 \pi}{\omF}$ non-impacting solution.
This periodic solution grazes the wall when its maximum $x$-value is zero.
It is a simple exercise to show from \eqref{eq:phip} that this occurs when $\cA = \cA_{\rm graz}(\omF)$, where
\begin{equation}
\cA_{\rm graz}(\omF) = \sqrt{\left( 1 - \omF^2 \right)^2 + 4 \zeta^2 \omF^2}.
\label{eq:Agraz}
\end{equation}
At grazing, the phase $z = \omF t ~{\rm mod}~ 2 \pi$ is $z = z_{\rm graz}$, where
\begin{align}
\sin \left( z_{\rm graz} \right) &= \frac{2 \zeta \omF}{\cA_{\rm graz}(\omF)}, &
\cos \left( z_{\rm graz} \right) &= \frac{1 - \omF^2}{\cA_{\rm graz}(\omF)}.
\end{align}
\label{le:oscGraz}

\subsection{Numerical bifurcation analysis}
\label{sub:2}

Motivated by physical experiments reported in \cite{PaIn10,InPa08b},
we fix $\zeta = 0.02$ and $\epsilon = 0.9$,
corresponding to relatively low damping and energy loss at impacts.
Fig.~\ref{fig:G} shows a numerically computed two-parameter bifurcation diagram
for the system at these parameter values.
The light green curve is the grazing bifurcation $\cA = \cA_{\rm graz}(\omF)$.
To the left of this curve (light grey) the non-impacting solution $\phi_p(t)$ is
an asymptotically stable solution of the full system.

The other shaded regions (dark grey)
are where the system has an asymptotically stable $p$-loop MPS for $p = 1,2,3$.
Each of these regions is bounded by three curves:
a curve PD where the MPS loses stability in a period-doubling bifurcation,
a curve SN where the MPS collides and annihilates with an unstable MPS of the same period,
and a curve GZ where the MPS is destroyed in a grazing bifurcation.
These curves were continued numerically using the approach of \cite{GhSi25b}
that employs a root-finding method to locate zeros of the VIVID function, see \S\ref{sub:vivid}.

\begin{figure}[b!]
\centering
\includegraphics[width=12cm]{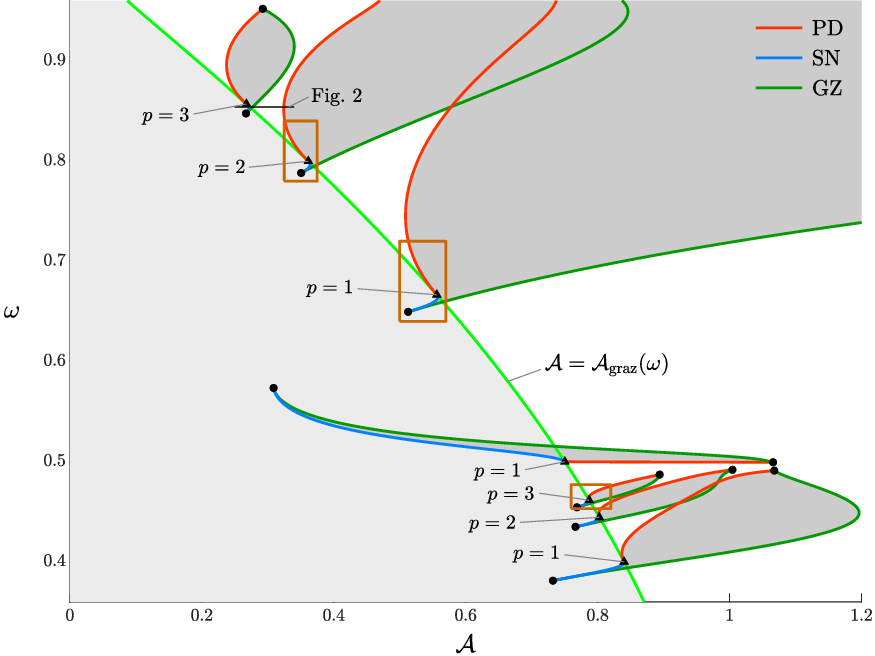}
\caption{
A two-parameter bifurcation diagram of the impact oscillator model
\eqref{eq:oscf}--\eqref{eq:oscResetLaw} with $\zeta = 0.02$ and $\epsilon = 0.9$.
The parameters on the axes are the forcing amplitude $\cA$, and the forcing frequency $\omF$.
The curve $\cA = \cA_{\rm graz}(\omega)$ is the grazing bifurcation \eqref{eq:Agraz}
of the non-impacting periodic solution.
The red, blue, and dark green curves are period-doubling, saddle-node, and grazing bifurcations
of $p$-loop MPSs for $p = 1,2,3$.
These curves meet at resonant grazing bifurcations (black triangles),
and other codimension-two points (black circles).
Magnified views of the diagram over the orange rectangles
are shown in Fig.~\ref{fig:Y}.
The line segment at $\omega = 0.854$ corresponds
to the one-parameter bifurcation diagram shown in Fig.~\ref{fig:Z}.
\label{fig:G}}
\end{figure}

\subsection{Parameter values associated with grazing}
\label{sub:oscGrazParams}

The saddle-node and period-doubling bifurcation curves
arise from points of resonance on the grazing curve $\cA = \cA_{\rm graz}(\omF)$.
To compute these points, we first observe that the impact oscillator model takes
the general form \eqref{eq:f}--\eqref{eq:resetLaw} with
\begin{align}
F(x,\dot{x},t) &= -2 \zeta \dot{x} - x - 1 + \cA \cos(\omF t), &
\Phi(y,z) &= \epsilon, &
\Psi(y,z) &= 0.
\nonumber
\end{align}
By evaluating \eqref{eq:phipsigamma}, we obtain
\begin{align}
\phi &= \epsilon, &
\psi &= 0, &
\gamma &= \omF^2,
\label{eq:oscphipsigamma}
\end{align}
and thus
\begin{equation}
\alpha = 1 + \epsilon,
\label{eq:oscalpha}
\end{equation}
by \eqref{eq:alpha}.
To apply the theory of \S\ref{sec:results}, we treat
$\mu = \cA - \cA_{\rm graz}(\omF)$ as the primary bifurcation parameter.

\begin{proposition}
The first derivatives of $P_{\rm global}$ at grazing, \eqref{eq:A} and \eqref{eq:b}, are given by
\begin{align}
A &= \re^{-\frac{2 \pi \zeta}{\omF}}
\begin{bmatrix} \cos \left( \frac{2 \pi \omDN}{\omF} \right)
+ \frac{\zeta}{\omDN} \,\sin \left( \frac{2 \pi \omDN}{\omF} \right) &
\frac{\omF}{\omDN} \,\sin \left( \frac{2 \pi \omDN}{\omF} \right) \\[1mm]
\frac{-1}{\omDN \omF} \,\sin \left( \frac{2 \pi \omDN}{\omF} \right) &
\cos \left( \frac{2 \pi \omDN}{\omF} \right)
- \frac{\zeta}{\omDN} \,\sin \left( \frac{2 \pi \omDN}{\omF} \right) \end{bmatrix}, \label{eq:oscA} \\
b &= \frac{1}{\cA_{\rm graz}(\omF)} \begin{bmatrix} 1 - a_{11} \\[1mm]
-a_{21} \end{bmatrix}, \label{eq:oscb}
\end{align}
where again we write $a_{ij}$ for the $(i,j)$-entry of $A$.
\label{pr:oscAb}
\end{proposition}

The formulas \eqref{eq:oscA} and \eqref{eq:oscb}
are derived in Appendix \ref{app:linearImpactOsc} by expanding the explicit solution \eqref{eq:flow} to first order.
The trace and determinant of \eqref{eq:oscA} are
\begin{align}
\tau &= 2 \re^{-\frac{2 \pi \zeta}{\omF}} \cos \left( \tfrac{2 \pi \omDN}{\omF} \right), &
\delta &= \re^{-\frac{4 \pi \zeta}{\omF}},
\label{eq:osctaudelta}
\end{align}
and by evaluating \eqref{eq:beta} we obtain
\begin{equation}
\beta = \frac{\delta - \tau + 1}{\cA_{\rm graz}(\omF)}.
\nonumber
\end{equation}

\subsection{Points of resonance}
\label{sub:4}

For $p = 1$, the codimension-two scenario occurs when $a_{12} = 0$.
By \eqref{eq:oscA}, $a_{12} = 0$ if and only if
\begin{equation}
\frac{\omDN}{\omF} = \frac{n}{2},
\label{eq:resonance1}
\end{equation}
for some $n \in \mathbb{Z}$.
In Fig.~\ref{fig:G}, the points on $\cA = \cA_{\rm graz}(\omF)$ labelled $p = 1$
correspond to \eqref{eq:resonance1} with $n = 3,4,5$:
specifically, $\omF = \frac{2 \omDN}{3} \approx 0.6665$,
$\omF = \frac{2 \omDN}{4} \approx 0.4999$,
and $\omF = \frac{2 \omDN}{5} \approx 0.3999$.

For $p \ge 2$, the codimension-two scenario occurs when $\tau = g_p(\delta)$,
where $g_p(\delta) = 2 \sqrt{\delta} \cos \left( \frac{\pi}{p} \right)$.
By \eqref{eq:osctaudelta}, $\tau = g_p(\delta)$ if and only if
\begin{equation}
\frac{\omDN}{\omF} = n \pm \frac{1}{2 p},
\label{eq:resonancep}
\end{equation}
for some $n \in \mathbb{Z}$.
In Fig.~\ref{fig:G}, the points on $\cA = \cA_{\rm graz}(\omF)$ labelled $p = 2,3$
correspond to $\frac{\omDN}{\omF} = n + \frac{1}{2 p}$ with $n = 1,2$.
Points with $\frac{\omDN}{\omF} = n - \frac{1}{2 p}$
do not produce the codimension-two phenomenon because in this case the MPSs are virtual.
This can be explained as follows. 
By substituting \eqref{eq:resonancep} into the $(1,2)$ entry of \eqref{eq:oscA}, we obtain
\begin{equation}
a_{12} = \frac{\pm \re^{\frac{-2 \pi \zeta}{\omF}}}
{n \pm \frac{1}{2 p}} \,\sin \left( \frac{\pi}{p} \right).
\nonumber
\end{equation}
By \eqref{eq:oscalpha}, if a positive sign is used in \eqref{eq:resonancep} then $a_{12} \alpha > 0$,
while if a negative sign is used in \eqref{eq:resonancep} then $a_{12} \alpha < 0$.
Theorem \ref{th:codim2p} requires $a_{12} \alpha > 0$;
if $a_{12} \alpha < 0$ then the MPSs are virtual
as can be inferred from calculations performed in \S\ref{sub:mainArguments}.

Pavlovskaia {\em et al}.~\cite{PaIn10} provide experimentally
computed bifurcation diagrams of a physical dynamic shaker.
They use $\omF$ as the primary bifurcation parameter
and obtain grazing at $\omF \approx 0.8$.
This is near the $p = 2$ resonance point $\frac{\omDN}{\omF} = n + \frac{1}{2 p}$ with $n = 1$,
specifically $\omF = \frac{4 \omDN}{5} \approx 0.7998$,
at which a stable two-loop MPS is generated.
This explains why their bifurcation diagram \cite[Figure 3]{PaIn10}
shows that a stable two-loop MPS is the dominant attractor just beyond the grazing bifurcation.

\subsection{Quadratic approximations to the saddle-node and period-doubling bifurcation curves}
\label{sub:5}

To apply Theorems \ref{th:codim21} and \ref{th:codim2p}, we use $\omF$ as the secondary parameter.
Specifically, we let $\eta = \omF - \omF^*$,
where we write $\omF^*$ for the value of $\omF$ at the resonant grazing bifurcation.

By further using the flow \eqref{eq:flow}, we obtain the following formulas
for $a_{12}'$, $\kappa_p'$, $s^\pm_p$ and the quadratic coefficients $c_{{\rm SN},p}$ and $c_{{\rm PD},p}$.
These formulas are derived in Appendix \ref{app:linearImpactOsc}.
For brevity we write
\begin{equation}
E_p = \re^{\frac{-2 \pi p \zeta}{\omF}},
\label{eq:Ep}
\end{equation}
and notice from \eqref{eq:osctaudelta} that $E_p = \delta^{\frac{p}{2}}$.

\begin{proposition}[resonant grazing, $p=1$]
Consider the impact oscillator \eqref{eq:oscf}--\eqref{eq:oscResetLaw}
with $\cA = \cA_{\rm graz}(\omF)$ and $\frac{\omDN}{\omF} = \frac{n}{2}$, where $n \ge 1$ is an integer.
If $n$ is odd, then
\begin{equation}
\begin{split}
a_{12}' &= \frac{2 \pi E_1}{\omF}, \\
s^+_1 &= -(1 - \epsilon E_1)^2, \\
s^-_1 &= (1 + \epsilon E_1)^2, \\
c_{{\rm SN},1} &= -\frac{2 \pi^2 A_{\rm graz}}{\omF^2}
\left( \frac{(1+\epsilon) E_1}
{\left( 1 - \epsilon E_1 \right) \left( 1 + E_1 \right)} \right)^2, \\
c_{{\rm PD},1} &= -\frac{\left( 1 - \epsilon E_1 \right)^2
\left( \left( 1 + \epsilon E_1 \right)^2 + 2 \left( 1 + \epsilon^2 E_1^2 \right) \right)}
{\left( 1 + \epsilon E_1 \right)^4} \,c_{{\rm SN},1},
\end{split}
\label{eq:impFormula1odd}
\end{equation}
while if $n$ is even, then
\begin{equation}
\begin{split}
a_{12}' &= -\frac{2 \pi E_1}{\omF}, \\
s^+_1 &= -(1 + \epsilon E_1)^2, \\
s^-_1 &= (1 - \epsilon E_1)^2, \\
c_{{\rm SN},1} &= -\frac{2 \pi^2 A_{\rm graz}}{\omF^2}
\left( \frac{(1+\epsilon) E_1}
{\left( 1 + \epsilon E_1 \right) \left( 1 - E_1 \right)} \right)^2, \\
c_{{\rm PD},1} &= -\frac{\left( 1 + \epsilon E_1 \right)^2
\left( \left( 1 - \epsilon E_1 \right)^2 + 2 \left( 1 + \epsilon^2 E_1^2 \right) \right)}
{\left( 1 - \epsilon E_1 \right)^4} \,c_{{\rm SN},1}.
\end{split}
\label{eq:impFormula1even}
\end{equation}
\label{pr:imp1}
\end{proposition}

\begin{proposition}[resonant grazing, $p \ge 2$]
For the impact oscillator \eqref{eq:oscf}--\eqref{eq:oscResetLaw}
with $\cA = \cA_{\rm graz}(\omF)$ and $\frac{\omDN}{\omF} = n + \frac{1}{2 p}$, where $n \ge 1$ and $p \ge 2$ are integers,
\begin{equation}
\begin{split}
\kappa_p' &= \frac{4 \pi \omDN}{\omF^2} \,\re^{-\frac{2 \pi \zeta}{\omF}} \sin \left( \tfrac{\pi}{p} \right), \\
s^+_1 &= -(1 - \epsilon E_p)^2, \\
s^-_1 &= (1 + \epsilon E_p)^2, \\
c_{{\rm SN},p} &= -\frac{2 \pi^2 p^2 A_{\rm graz}}{\omF^2}
\left( \frac{(1+\epsilon) E_p}
{\left( 1 - \epsilon E_p \right) \left( 1 + E_p \right)} \right)^2, \\
c_{{\rm PD},p} &= -\frac{\left( 1 - \epsilon E_p \right)^2
\left( \left( 1 + \epsilon E_p \right)^2 + 2 \left( 1 + \epsilon^2 E_p^2 \right) \right)}
{\left( 1 + \epsilon E_p \right)^4} \,c_{{\rm SN},p}.
\end{split}
\label{eq:impFormulap}
\end{equation}
\label{pr:impp}
\end{proposition}

To illustrate Propositions \ref{pr:imp1} and \ref{pr:impp},
Fig.~\ref{fig:Y} shows magnifications of Fig.~\ref{fig:G} about three points of resonance.
In these magnifications we have overlaid the quadratic approximations $\mu = c_{{\rm SN},p} \eta^2$ and $\mu = c_{{\rm PD},p} \eta^2$ (dashed).
This was achieved by evaluating \eqref{eq:impFormula1odd} and \eqref{eq:impFormulap}, and inverting the coordinate change
$(\mu,\eta) = \left( \cA - \cA_{\rm graz}(\omF), \omF - \omF^* \right)$.
As expected, the quadratic approximations provide a good fit to the bifurcation curves
as they emanate from the codimension-two points.

\begin{figure}[b!]
\centering
\includegraphics[width=15.6cm]{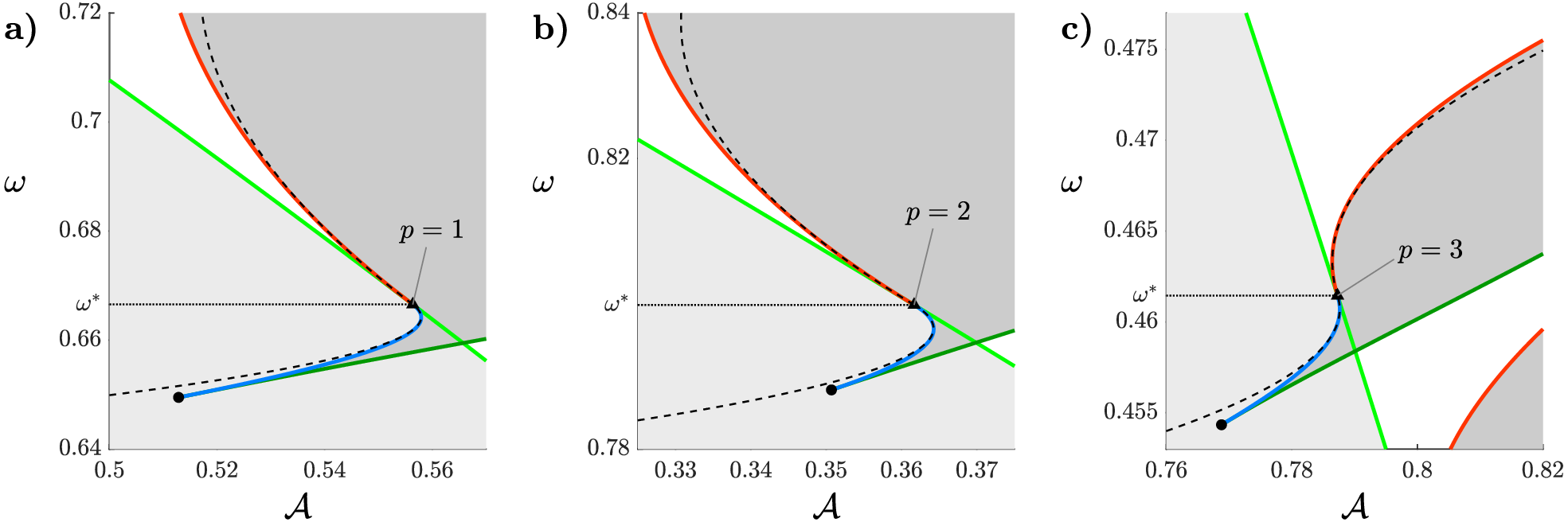}
\caption{
Magnifications of Fig.~\ref{fig:G};
the dashed curves show $\mu = c_{{\rm SN},p} \eta^2$ and $\mu = c_{{\rm PD},p} \eta^2$
where $c_{{\rm SN},p}$ and $c_{{\rm PD},p}$
are given by Propositions \ref{pr:imp1} and \ref{pr:impp}.
To four significant figures, $c_{{\rm SN},1} = -282.4$ and $c_{{\rm PD},1} = 12.15$ for the left plot,
$c_{{\rm SN},2} = -244.5$ and $c_{{\rm PD},2} = 21.35$ for the middle plot,
and $c_{{\rm SN},3} = -613.4$ and $c_{{\rm PD},3} = 249.1$ for the right plot.
\label{fig:Y}}
\end{figure}

Propositions \ref{pr:imp1} and \ref{pr:impp} show that if $\epsilon \ne \frac{1}{E_p}$,
then $c_{{\rm SN},p} < 0$ and $c_{{\rm PD},p} > 0$ for all $p \ge 1$.
In this case the saddle-node and period-doubling bifurcation curves exist on
different sides of $\cA = \cA_{\rm graz}$, as in Fig.~\ref{fig:G}.
Furthermore, in this case $s_p^+$ and $s_p^-$ have different signs,
so by Theorems \ref{th:codim1p} and \ref{th:codim2p}
the saddle-node and period-doubling bifurcation curves grow in opposite directions out of the
resonant grazing point, as seen for all points of resonance indicated in Fig.~\ref{fig:G}.

\section{Calculations for generic grazing bifurcations}
\label{sec:generic}

In this section we prove Theorems \ref{th:codim11} and \ref{th:codim1p}.
First in \S\ref{sub:discMap} we compute the components of the discontinuity map,
then in \S\ref{sub:global} characterise the $p^{\rm th}$ iterate of the global map to first order.
In \S\ref{sub:vivid} we define the VIVID function and compute its derivatives,
then in \S\ref{sub:mainArguments} combine the computations to verify
Theorems \ref{th:codim11} and \ref{th:codim1p}.
Throughout this section we ignore the second parameter $\eta$,
and write $\omF$ in place of $\omF(0)$.

\subsection{The discontinuity map}
\label{sub:discMap}

We first introduce some additional notation.
Suppose an orbit of \eqref{eq:f}--\eqref{eq:resetLaw} reaches the impacting surface
at a point $(0,y_{\rm imp},z_{\rm imp})$ on the incoming set, i.e.~$y_{\rm imp} > 0$, see Fig.~\ref{fig:Bb}.
This orbit subsequently departs the impacting surface from the point $(0,y_{\rm rec},z_{\rm rec})$, where
\begin{equation}
(y_{\rm rec},z_{\rm rec}) = R(y_{\rm imp},z_{\rm imp};\mu),
\label{eq:}
\end{equation}
belongs to the outgoing set, i.e.~$y_{\rm rec} < 0$.

\begin{figure}[b!]
\centering
\includegraphics[width=12cm]{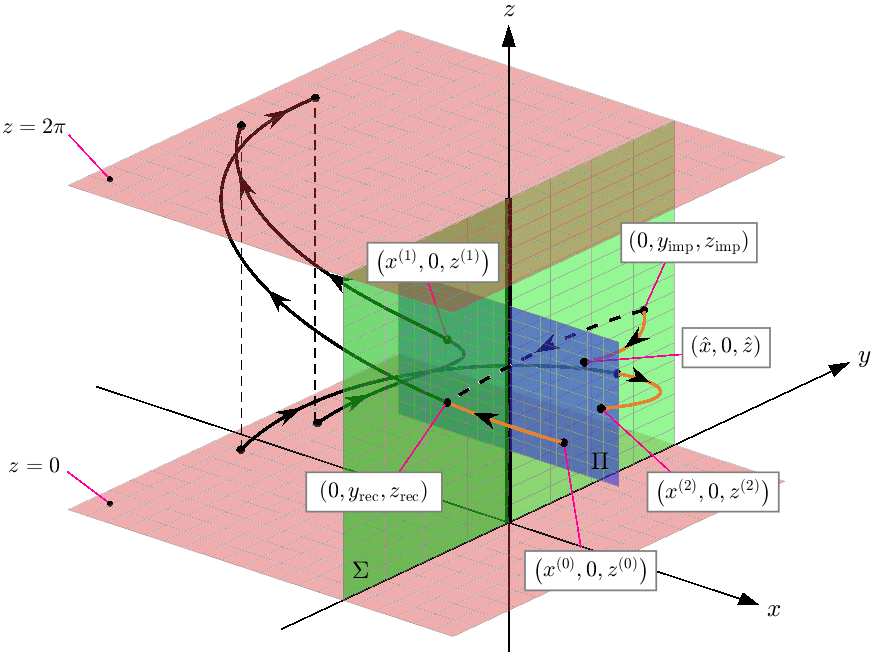}
\caption{
A sketch indicating points defined in the text where an orbit intersects $\Sigma$ and $\Pi$.
\label{fig:Bb}}
\end{figure}

Now consider a smooth extension of $F$ into $x > 0$.
We evolve under the ODEs \eqref{eq:f} into $x > 0$
from the impacting point forwards in time until reaching $\Pi$ at a point $(\hat{x},0,\hat{z})$,
and from the recoil point backwards in time until reaching $\Pi$ at a point $\left( x^{(0)}, 0, z^{(0)} \right)$.
These points are {\em virtual} because they do not belong to orbits of the full system \eqref{eq:f}--\eqref{eq:resetLaw}.

Let $(x',z') = P_{\rm virt}(y,z;\mu)$ be the map that takes points from $\Sigma$ to $\Pi$ following \eqref{eq:f}, so
\begin{align}
\left( \hat{x}, \hat{z} \right) &= P_{\rm virt}(y_{\rm imp},z_{\rm imp};\mu), &
\left( x^{(0)}, z^{(0)} \right) &= P_{\rm virt}(y_{\rm rec},z_{\rm rec};\mu).
\label{eq:Pvirt2}
\end{align}
The map is evaluated by evolving forwards in time from $\Sigma$ if $y > 0$,
and backwards in time from $\Sigma$ if $y < 0$ (if $y = 0$ the evolution time is zero and $(x',z') = (0,z)$).
This stipulation on the direction of time ensures that the evolution time is small when
the value of $y$ is small.
By considering smooth extensions of $\Phi$ and $\Psi$ into $y < 0$,
the above points are well defined for small $y_{\rm imp} < 0$.

If $F$ is $C^k$, then $P_{\rm virt}$ is $C^k$
at any $(y_{\rm imp},z_{\rm imp};\mu)$ for which the associated orbit hits $\Pi$ transversally
(this is a consequence of the implicit function theorem \cite{Me07}).
With Assumption \ref{ass:grazing}, transversality is satisfied at grazing,
thus $P_{\rm virt}$ is $C^k$ in a neighbourhood of $(0,z_{\rm graz};0)$.

Given $(y_{\rm imp},z_{\rm imp};\mu)$ near $(0,z_{\rm graz};0)$,
the following result provides asymptotic formulas for $(\hat{x},\hat{z})$ and $\left( x^{(0)}, z^{(0)} \right)$
which involve the constants defined in \eqref{eq:phipsigamma}.
A derivation is provided below.
Similar calculations can be found in \cite{No01,MoDe01}.

\begin{lemma}
We have
\begin{equation}
\begin{split}
\hat{x} &= \frac{y_{\rm imp}^2}{2 \gamma}
+ \cO \left( \left( |y_{\rm imp}| + |z_{\rm imp} - z_{\rm graz}| + |\mu| \right)^3 \right), \\
\hat{z} &= \frac{\omF y_{\rm imp}}{\gamma} + z_{\rm imp}
+ \cO \left( \left( |y_{\rm imp}| + |z_{\rm imp} - z_{\rm graz}| + |\mu| \right)^2 \right),
\end{split}
\label{eq:xHatzHat}
\end{equation}
and
\begin{equation}
\begin{split}
x^{(0)} &= \frac{\phi^2 y_{\rm imp}^2}{2 \gamma}
+ \cO \left( \left( |y_{\rm imp}| + |z_{\rm imp} - z_{\rm graz}| + |\mu| \right)^3 \right), \\
z^{(0)} &= \frac{\omF(1-\alpha)}{\gamma} \,y_{\rm imp} + z_{\rm imp}
+ \cO \left( \left( |y_{\rm imp}| + |z_{\rm imp} - z_{\rm graz}| + |\mu| \right)^2 \right).
\end{split}
\label{eq:x0z0}
\end{equation}
\label{le:discMapComponents}
\end{lemma}

\begin{proof}
The ODEs \eqref{eq:f} can be written as
\begin{equation}
\begin{bmatrix} \dot{x} \\ \dot{y} \\ \dot{z} \end{bmatrix}
= \begin{bmatrix} y \\ -\gamma + \cO \left( |x| + |y| + |z - z_{\rm graz}| + |\mu| \right) \\ \omF + \cO(\mu) \end{bmatrix}.
\nonumber
\end{equation}
Thus the orbit of \eqref{eq:f} through $(x,y,z) = (0,y_{\rm imp},z_{\rm imp})$ is given by	
\begin{equation}
\begin{split}
x(t) &= y_{\rm imp} t - \frac{\gamma t^2}{2} + \cO \left( \left( |y_{\rm imp}| + |z_{\rm imp}-z_{\rm graz}| + |\mu| + |t| \right)^3 \right), \\
y(t) &= y_{\rm imp} - \gamma t + \cO \left( \left( |y_{\rm imp}| + |z_{\rm imp}-z_{\rm graz}| + |\mu| + |t| \right)^2 \right), \\
z(t) &= z_{\rm imp} + \omF t + \cO \left( \left( |\mu| + |t| \right)^2 \right).
\end{split}
\nonumber
\end{equation}
Since $\gamma \ne 0$, we can solve $y(T) = 0$ for the evolution time $T$, resulting in
\begin{equation}
T(y_{\rm imp},z_{\rm imp};\mu) = \frac{y_{\rm imp}}{\gamma} + \cO \left( \left( |y_{\rm imp}| + |z_{\rm imp}-z_{\rm graz}| + |\mu| \right)^2 \right),
\nonumber
\end{equation}
and by substituting this into $x(t)$ and $z(t)$ we obtain \eqref{eq:xHatzHat}.
The reset law \eqref{eq:R} can be written as
\begin{equation}
\begin{bmatrix} y \\ z \end{bmatrix}
\mapsto \begin{bmatrix} -y \big( \phi + \cO \left( |y| + |z - z_{\rm graz}| + |\mu| \right) \big) \\
z + y \big( \psi + \cO \left( |y| + |z - z_{\rm graz}| + |\mu| \right) \big)
\end{bmatrix},
\nonumber
\end{equation}
thus
\begin{align}
y_{\rm rec} &= -\phi y_{\rm imp} + \cO \left( \left( |y_{\rm imp}| + |z_{\rm imp}-z_{\rm graz}| + |\mu| \right)^2 \right), \nonumber \\
z_{\rm rec} &= \psi y_{\rm imp} + z_{\rm imp} + \cO \left( \left( |y_{\rm imp}| + |z_{\rm imp}-z_{\rm graz}| + |\mu| \right)^2 \right). \nonumber
\end{align}
By using these in place of $y_{\rm imp}$ and $z_{\rm imp}$ in \eqref{eq:xHatzHat},
and inserting the formula \eqref{eq:alpha} for $\alpha$, we obtain \eqref{eq:x0z0}.
\end{proof}

For $x > 0$, the inverse $P_{\rm virt}^{-1}$
has two values: one on the incoming set $y > 0$, and one on the outgoing set $y < 0$.
Throughout this paper we always take the value on the incoming set.
Then
\begin{equation}
P_{\rm disc} = P_{\rm virt} \circ R \circ P_{\rm virt}^{-1},
\label{eq:Pdisc}
\end{equation}
is well-defined, and is the {\em discontinuity map}
that takes $(\hat{x},\hat{z})$ to $\left( x^{(0)}, z^{(0)} \right)$.
By inverting \eqref{eq:xHatzHat} and substituting the result into \eqref{eq:x0z0}, we obtain
\begin{align}
x^{(0)} &= \phi_0^2 \hat{x} + \cdots, \nonumber \\
z^{(0)} &= -\frac{\alpha \omF \sqrt{2}}{\sqrt{\gamma}} \sqrt{\hat{x}} + \hat{z} + \cdots,
\end{align}
to leading order.
This shows that $P_{\rm disc}$ contains a square-root term whose sign is controlled by the sign of $\alpha$.

\subsection{Iterates of the global map.}
\label{sub:global}

Define
\begin{align}
S_j &= \sum_{k=1}^j \lambda_1^{j-k} \lambda_2^{k-1}, &
T_p &= \sum_{j=1}^{p-1} S_j \,.
\label{eq:SpTp}
\end{align}
where $\lambda_1$ and $\lambda_2$ are the eigenvalues of $A$, see \eqref{eq:eigs}.
By \eqref{eq:Hp},
$H_p(\tau,\delta) = \sum_{j=1}^{p-1} \frac{S_j}{\delta^{j-1}}$ (since $\lambda_1 \lambda_2 = \delta$).
Moreover,
\begin{equation}
H_p(\tau,\delta) = \frac{S_{p-1} T_p - S_p T_{p-1}}{\delta^{p-2}},
\label{eq:Hp9}
\end{equation}
for all $p \ge 1$.
By composing the first-order expression \eqref{eq:Pglobal} of $P_{\rm global}$ with itself $p$ times, we obtain
\begin{equation}
P_{\rm global}^p(x,z;\mu) = \begin{bmatrix} 0 \\ z_{\rm graz} \end{bmatrix}
+ A^p \begin{bmatrix} x \\ z - z_{\rm graz} \end{bmatrix} + \begin{bmatrix} b_{1p} \\ b_{2p} \end{bmatrix} \mu
+ \cO \left( \left( |x| + |z-z_{\rm graz}| + |\mu| \right)^2 \right),
\label{eq:Pglobalp}
\end{equation}
where
\begin{align}
A^p &= S_{p+1} I + S_p \begin{bmatrix}
-a_{22} & a_{12} \\
a_{21} & -a_{11}
\end{bmatrix}, \label{eq:Ap} \\
\begin{bmatrix} b_{1p} \\ b_{2p} \end{bmatrix} &= 
S_p \begin{bmatrix} b_1 \\ b_2 \end{bmatrix}
+ T_p \begin{bmatrix} \beta \\ a_{21} b_1 + (1-a_{11}) b_2 \end{bmatrix}. \label{eq:b1pb2p}
\end{align}
The formulas \eqref{eq:Ap} and \eqref{eq:b1pb2p}
were derived via Sylvester's formula \cite{Do15},
and can be verified by induction on $p$.
The following result is proved in Appendix \ref{app:someProofs}.

\begin{lemma}
Let $p \ge 2$ and $(\tau,\delta) \in \cT$.
\begin{enumerate}
\item[(a)]
If $\tau > h_p(\delta)$, then ${\rm sgn}(S_p) = {\rm sgn}(\tau - g_p(\delta))$.
\item[(b)]
If $\tau = g_p(\delta)$, then $S_{p+1} = -\delta^{\frac{p}{2}}$ and $T_p = \frac{1 + \delta^{\frac{p}{2}}}{\delta - \tau + 1}$.
\item[(c)]
If $\tau \ge g_{\frac{p}{2}}(\delta)$, then $T_p > 0$.
\end{enumerate}
\label{le:SpTp}
\end{lemma}

\subsection{The VIVID function}
\label{sub:vivid}

Given $p \ge 1$, for all $j = 1,2,\ldots,p$ let
\begin{equation}
\left( x^{(j)}, z^{(j)} \right) = P_{\rm global} \left( x^{(j-1)}, z^{(j-1)}; \mu \right),
\label{eq:xjzj}
\end{equation}
as indicated in Fig.~\ref{fig:Bb} for $p=2$.
These points correspond to a $p$-loop MPS if $\left( x^{(p)}, z^{(p)} \right) = (\hat{x},\hat{z})$.
In this case $(\hat{x},\hat{z})$ is a fixed point of $P_{\rm global}^p \circ P_{\rm disc}$,
so can be found by solving for fixed points of this map.
However, we require smoothness, so instead search for zeros of the VIVID function
\begin{equation}
V(y_{\rm imp}, z_{\rm imp}; \mu) = \left( x^{(p)}, z^{(p)} \right) - \left( \hat{x}, \hat{z} \right).
\label{eq:V}
\end{equation}
Notice $V = P_{\rm global}^p \circ P_{\rm virt} \circ R - P_{\rm virt}$
is comprised of $C^k$ functions, so is $C^k$ and well-defined in a neighbourhood of the grazing bifurcation.

By Assumption \ref{ass:grazing}, $V(0,z_{\rm graz};0) = (0,0)$.
Write $V(y,z;\mu) = (V_1,V_2)$ and define
\begin{align}
J &= \begin{bmatrix} \frac{\partial V_1}{\partial y} & \frac{\partial V_1}{\partial z} \\[1mm]
\frac{\partial V_2}{\partial y} & \frac{\partial V_2}{\partial z} \end{bmatrix}
\Bigg|_{(y,z;\mu) = (0,z_{\rm graz};0)}, &
K &= \begin{bmatrix} \frac{\partial V_1}{\partial z} & \frac{\partial V_1}{\partial \mu} \\[1mm]
\frac{\partial V_2}{\partial z} & \frac{\partial V_2}{\partial \mu} \end{bmatrix}
\Bigg|_{(y,z;\mu) = (0,z_{\rm graz};0)}.
\label{eq:Jac1Jac2}
\end{align}
To find zeros of $V$, we use the implicit function theorem.
This requires the invertibility of $J$ or $K$.

\begin{lemma}
For all $p \ge 1$, the determinants of \eqref{eq:Jac1Jac2} are given by
\begin{align}
\det(J) &= \frac{\alpha a_{12} \omF S_p}{\gamma}, &
\det(K) &= (1 - S_{p+1} + a_{11} S_p) b_{1p} + a_{12} S_p b_{2p} \,.
\label{eq:detJac1detJac2}
\end{align}
Moreover, if $\tau \ne \delta + 1$ then
\begin{equation}
\det(K) = \frac{\left( 1 - \lambda_1^p \right) \left( 1 - \lambda_2^p \right) \beta}{(1 - \lambda_1)(1 - \lambda_2)}.
\label{eq:detJac2}
\end{equation}
\label{le:dets}
\end{lemma}

\begin{proof}
By combining \eqref{eq:xHatzHat}, \eqref{eq:x0z0}, \eqref{eq:Pglobalp}, and \eqref{eq:Ap} we obtain
\begin{align}
V_1 &= \frac{a_{12} \omF (1-\alpha) S_p}{\gamma} \,y_{\rm imp}
+ a_{12} S_p (z_{\rm imp} - z_{\rm graz}) + b_{1p} \mu \nonumber \\
&\quad+ \cO \left( \left( |y_{\rm imp}| + |z_{\rm imp}-z_{\rm graz}| + |\mu| \right)^2 \right), \label{eq:V1} \\
V_2 &= \frac{\omF}{\gamma} \left( (1-\alpha) (S_{p+1} - a_{11} S_p) - 1 \right) y_{\rm imp}
+ (S_{p+1} - a_{11} S_p - 1)(z_{\rm imp} - z_{\rm graz}) + b_{2p} \mu \nonumber \\
&\quad+ \cO \left( \left( |y_{\rm imp}| + |z_{\rm imp}-z_{\rm graz}| + |\mu| \right)^2 \right), \label{eq:V2}
\end{align}
using also the formula \eqref{eq:alpha} for $\alpha$.
Thus we read off
\begin{align}
J &= \begin{bmatrix}
\frac{a_{12} \omF (1-\alpha) S_p}{\gamma} & a_{12} S_p \\
\frac{\omF}{\gamma} \left( (1-\alpha) (S_{p+1} - a_{11} S_p) - 1 \right) & S_{p+1} - a_{11} S_p - 1
\end{bmatrix}, &
K &= \begin{bmatrix}
a_{12} S_p & b_{1p} \\
S_{p+1} - a_{11} S_p - 1 & b_{2p}
\end{bmatrix}, \nonumber
\end{align}
and by evaluating the determinants of these matrices we obtain \eqref{eq:detJac1detJac2}.
If $\tau \ne \delta + 1$, then $S_p = \frac{\lambda_1^p - \lambda_2^p}{\lambda_1 - \lambda_2}$,
see \eqref{eq:Sp2},
and by combining this with \eqref{eq:b1pb2p} we obtain \eqref{eq:detJac2} after algebraic simplification.
\end{proof}

\subsection{Main arguments for generic grazing bifurcations}
\label{sub:mainArguments}

\begin{proof}[Proof of Theorem \ref{th:codim11}]
Putting $p = 1$ into \eqref{eq:detJac1detJac2} gives
$\det(J) = \frac{\alpha a_{12} \omF}{\gamma} \ne 0$, because $S_1 = 1$.
Also $\tau \ne \delta + 1$, because $(\tau,\delta) \in \cT$,
so $\det(K) = \beta \ne 0$ by \eqref{eq:detJac2}.
We now use \eqref{eq:V1} and \eqref{eq:V2} to solve $V(y;z;\mu) = (0,0)$ for $y$ and $z$ in terms $\mu$.
Since $V$ is $C^3$ and $\det(J) \ne 0$,
by the implicit function theorem there exists a unique $C^3$ solution
\begin{align}
y^*(\mu) &= \frac{\gamma \beta}{\alpha a_{12} \omF} \,\mu + \cO \left( \mu^2 \right), \label{eq:yStar1codim1proof} \\
z^*(\mu) &= z_{\rm graz} + \frac{1}{\alpha a_{12}} \big( \left( a_{22} (1-\alpha) - 1 \right) b_1
- a_{12} (1-\alpha) b_2 \big) \mu + \cO \left( \mu^2 \right), \label{eq:zStar1codim1proof}
\end{align}
for small $\mu \in \mathbb{R}$.
This corresponds to a one-loop MPS if $y^*(\mu) > 0$ and $\mu$ is sufficiently small.
By assumption $\alpha > 0$, $\omF > 0$, and $\gamma > 0$.
Thus if $a_{12} \beta > 0$, then $y^*(\mu) > 0$ for small $\mu > 0$,
while if $a_{12} \beta < 0$, then $y^*(\mu) > 0$ for small $\mu < 0$.
\end{proof}

\begin{proof}[Proof of Theorem \ref{th:codim1p}]
The proof is completed in three steps.
In Step 1 we compute a zero of the VIVID function corresponding to a $p$-loop MPS,
then map it under $R$, $P_{\rm virt}$, and $P_{\rm global}$
to identify all points where the $p$-loop MPS intersects the Poincar\'e section $\Pi$.
In Step 2 we remove higher-order terms,
so that in Step 3 admissibility can be characterised through
brute-force algebraic computations and the geometric properties of a linear map.

\myStep{1}{Calculate points.}
Suppose $\kappa_p \ne 0$, i.e.~$\tau \ne g_p(\delta)$.
By Lemma \ref{le:SpTp}(a), ${\rm sgn}(S_p) = {\rm sgn}(\kappa_p)$,
and so $\det(J) \ne 0$ by \eqref{eq:detJac1detJac2}.
By \eqref{eq:detJac2}, ${\rm sgn}(\det(K)) = {\rm sgn}(\beta)$ 
because $|\lambda_1|, |\lambda_2| < 1$ in view of $(\tau,\delta) \in \cT$.

We now use \eqref{eq:V1} and \eqref{eq:V2} to solve $V(y;z;\mu) = (0,0)$ for $y$ and $z$ in terms $\mu$.
Since $V$ is $C^3$ and $\det(J) \ne 0$,
by the implicit function theorem there exists a unique $C^3$ solution
\begin{align}
y^*(\mu) &= \frac{\gamma \det(K)}{\alpha a_{12} \omF S_p} \,\mu + \cO \left( \mu^2 \right), \label{eq:yStarpcodim1proof} \\
z^*(\mu) &= z_{\rm graz} + \frac{1}{\alpha a_{12} S_p} \big( \left( (S_{p+1} - a_{11} S_p) (1-\alpha) - 1 \right) b_{1p}
- a_{12} S_p (1-\alpha) b_{2p} \big) \mu + \cO \left( \mu^2 \right), \label{eq:zStarpcodim1proof}
\end{align}
for small $\mu \in \mathbb{R}$.

Let $\left( x^{(0)}(\mu), z^{(0)}(\mu) \right) = P_{\rm virt} \left( R \left( y^*(\mu), z^*(\mu); \mu \right); \mu \right)$,
and define $\left( x^{(j)}(\mu), z^{(j)}(\mu) \right)$ by \eqref{eq:xjzj} for all $j = 1,2,\ldots,p$.
By inserting \eqref{eq:yStarpcodim1proof}--\eqref{eq:zStarpcodim1proof} into \eqref{eq:x0z0}, we obtain
\begin{align}
x^{(0)}(\mu) &= \frac{\gamma \phi^2 \det(K)^2}{2 \alpha^2 a_{12}^2 \omF^2 S_p^2} \,\mu^2
+ \cO \left( \mu^3 \right), \label{eq:x0codim1proof} \\
z^{(0)}(\mu) &= z_{\rm graz} - \frac{b_{1p}}{a_{12} S_p} \,\mu + \cO \left( \mu^2 \right), \label{eq:z0codim1proof}
\end{align}
using the formula \eqref{eq:detJac1detJac2} for $\det(K)$ to achieve simplification in \eqref{eq:z0codim1proof}.
Since $\left( x^{(p)}(\mu), z^{(p)}(\mu) \right) = P_{\rm virt} \left( y^*(\mu), z^*(\mu); \mu \right)$,
we can insert \eqref{eq:yStarpcodim1proof}--\eqref{eq:zStarpcodim1proof} into \eqref{eq:xHatzHat} to obtain
\begin{align}
x^{(p)}(\mu) &= \frac{\gamma \det(K)^2}{2 \alpha^2 a_{12}^2 \omF^2 S_p^2} \,\mu^2
+ \cO \left( \mu^3 \right), \label{eq:xpcodim1proof} \\
z^{(p)}(\mu) &= z_{\rm graz} + \frac{\det(K) - b_{1p}}{a_{12} S_p} \,\mu + \cO \left( \mu^2 \right), \label{eq:zpcodim1proof}
\end{align}
using again \eqref{eq:detJac1detJac2}.

\myStep{2}{Reduce to a linear map.}
In order for $\left( y^*(\mu), z^*(\mu) \right)$ to correspond to a $p$-loop MPS of \eqref{eq:f}--\eqref{eq:resetLaw},
we require $y^*(\mu) > 0$, and $x^{(j)}(\mu) < 0$
for all $j = 1,2,\ldots,p-1$ so that the periodic solution only hits the wall once per period.
By assumption, $a_{12} \alpha > 0$, $\omF > 0$, and $\gamma > 0$.
Thus by \eqref{eq:yStarpcodim1proof} the sign of $\frac{d y^*}{d \mu}(0)$
equals the sign of $\det(K) S_p$, which equals the sign of $\beta \kappa_p$.
Thus if $\beta \kappa_p > 0$, then $y^*(\mu) > 0$ for small $\mu > 0$,
while if $\beta \kappa_p < 0$, then $y^*(\mu) > 0$ for small $\mu < 0$.

It remains to show that for all $j = 1,2,\ldots,p-1$
the sign of $\frac{d x^{(j)}}{d \mu}(0)$ is opposite to the sign of $\beta \kappa_p$.
For all $j = 0,1,\ldots,p$, let
\begin{equation}
\left( u^{(j)}, w^{(j)} \right) = \frac{S_p}{\beta} \left( \frac{d x^{(j)}}{d \mu}(0), \frac{d z^{(j)}}{d \mu}(0) \right).
\label{eq:ujwj}
\end{equation}
It remains to show $u^{(j)} < 0$ for all $j = 1,2,\ldots,p-1$ (because ${\rm sgn}(S_p) = {\rm sgn}(\kappa_p)$).

In view of \eqref{eq:Pglobal} and \eqref{eq:xjzj},
$\left( u^{(j)}, w^{(j)} \right) = Q \left( u^{(j-1)}, w^{(j-1)} \right)$ for all $j = 1,2,\ldots,p$,
where $Q$ is the affine map
\begin{equation}
Q(u,w) = A \begin{bmatrix} u \\ w \end{bmatrix} + \frac{S_p}{\beta} \,b.
\label{eq:affineMap}
\end{equation}
Thus
\begin{equation}
\begin{bmatrix} u^{(j)} \\ w^{(j)} \end{bmatrix} =
A^j \begin{bmatrix} u^{(0)} \\ w^{(0)} \end{bmatrix} + \frac{S_p}{\beta} \begin{bmatrix} b_{1j} \\ b_{2j} \end{bmatrix},
\label{eq:affineMapj}
\end{equation}
for all $j = 1,2,\ldots,p$.
By \eqref{eq:x0codim1proof} and \eqref{eq:z0codim1proof}, $u^{(0)} = 0$ and $w^{(0)} = -\frac{b_{1p}}{a_{12} \beta}$.
By inserting these and the formulas \eqref{eq:Ap} and \eqref{eq:b1pb2p} into \eqref{eq:affineMapj} we obtain
\begin{equation}
u^{(j)} = S_p T_j - S_j T_p \,,
\label{eq:ujcodim1proof}
\end{equation}
after simplification.

\myStep{3}{Verify admissibility.}
We first evaluate \eqref{eq:ujcodim1proof} with $j=1$ and $j=p-1$.
Since $S_1 = 1$ and $T_1 = 0$,
\begin{align}
u^{(1)} &= -T_p \,, \label{eq:u1codim1proof} \\
u^{(p-1)} &= -\delta^{p-2} H_p(\tau,\delta), \label{eq:upm1codim1proof}
\end{align}
using also \eqref{eq:Hp9}.
By assumption, $\tau > h_p(\delta)$, so $H_p(\tau,\delta) > 0$,
by the definition of $h_p$ (see \S\ref{sub:mps}), and hence $u^{(p-1)} < 0$.
Also $\tau > g_{\frac{p}{2}}(\delta)$, by Lemma \ref{le:admissibilityBoundary},
so $T_p > 0$, by Lemma \ref{le:SpTp}(c), and hence $u^{(1)} < 0$.

It remains to show $u^{(j)} < 0$ for all $2 \le j \le p-2$ when $p \ge 3$.
First consider the boundary case $\tau = \delta+1$ (with $0 < \delta < 1$).
Here $\lambda_1 = 1$,
$\lambda_2 = \delta$, and
$S_j = \frac{1 - \delta^j}{1 - \delta}$ for each $j$.
So from \eqref{eq:SpTp}, and the formula
for a truncated geometric series\footnote{
$\sum_{k=0}^{n-1} r^k = \frac{1 - r^n}{1 - r}$
}, we obtain
$T_p = \frac{p (1 - \delta) - \left( 1 - \delta^p \right)}{(1-\delta)^2}$.
Then by \eqref{eq:ujcodim1proof} 
\begin{equation}
u^{(j)} = \frac{j \left( 1 - \delta^p \right) - p \left( 1 - \delta^j \right)}{(1-\delta)^2},
\label{eq:ujeig1}
\end{equation}
which is negative for all $0 < j < p$ (Lemma \ref{le:ujNumerator}).

\begin{figure}[b!]
\centering
\includegraphics[width=8cm]{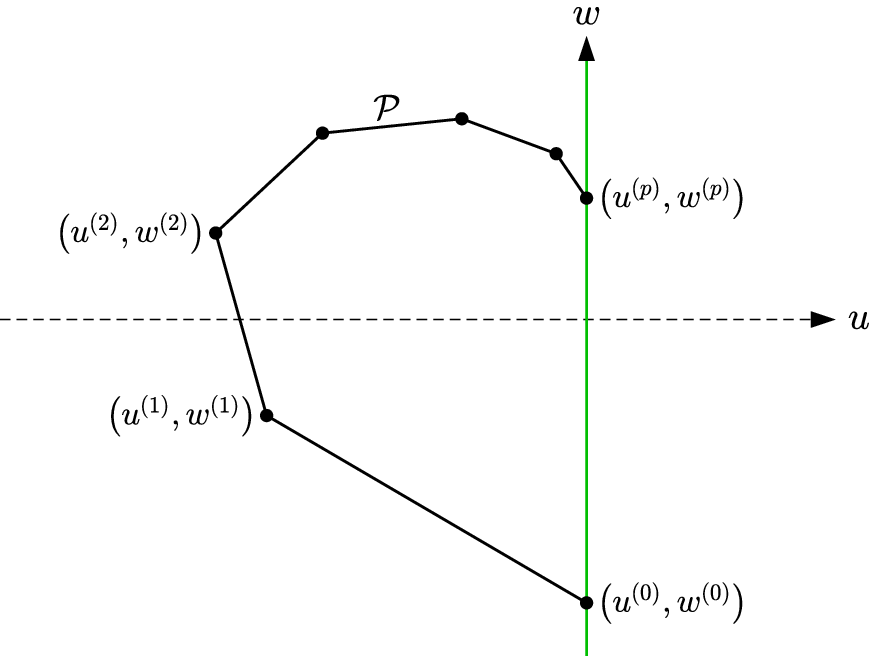}
\caption{
A sketch of the path $\cP$ introduced in Step 3 of the proof of Theorem \ref{th:codim1p}.
\label{fig:I}}
\end{figure}

Let $\cP$ be the path obtained by linearly connecting
$\left( u^{(0)}, w^{(0)} \right)$ to $\left( u^{(1)}, w^{(1)} \right)$ to $\left( u^{(2)}, w^{(2)} \right)$
and so on up to $\left( u^{(p)}, w^{(p)} \right)$, see Fig.~\ref{fig:I}.
Since $u^{(0)} = u^{(p)} = 0$, $\cP$ starts and ends on the $w$-axis.
Since $Q$ is affine, $\cP$ has no self intersections and is convex
(i.e.~the convex hull of $\cP$ is equal to the filled polygon
with vertices $\left( u^{(0)}, w^{(0)} \right), \ldots, \left( u^{(p)}, w^{(p)} \right)$).

We now fix $0 < \delta < 1$ and decrease the value of $\tau$ from $\delta + 1$.
As we do so, $\cP$ varies continuously.
Let $\tau^*$ be the first value of $\tau$ at which $u^{(\ell)} = 0$, for some $1 \le \ell \le p-1$.
This cannot occur with $2 \le \ell \le p-2$, for then $\cP$ would be non-convex when $\tau = \tau^*$.
But $u^{(1)} < 0$ and $u^{(p-1)} < 0$ when $\tau > h_p(\delta)$, thus $\tau^* \le h_p(\delta)$.
Thus for any $h_p(\delta) < \tau < \delta + 1$, we have $u^{(j)} < 0$ for all $2 \le j \le p-2$ as required.
\end{proof}

\section{Calculations for resonant grazing bifurcations}
\label{sec:resonant}

In this section we prove Theorems \ref{th:codim21} and \ref{th:codim2p}
now working with both parameters $\mu$ and $\eta$.

\subsection{Stability of maximal periodic solutions}

Consider a $p$-loop MPS with values $y_{\rm imp}$ and $z_{\rm imp}$ at impact.
The stability multipliers of this solution are the eigenvalues of
\begin{equation}
U(y_{\rm imp},z_{\rm imp};\mu,\eta) = \rD \left( P_{\rm global}^p \circ P_{\rm disc} \right)
\left( \hat{x}, \hat{z}; \mu, \eta \right),
\label{eq:U}
\end{equation}
where $\left( \hat{x}, \hat{z} \right) = P_{\rm virt}(y_{\rm imp},z_{\rm imp};\mu,\eta)$.
These eigenvalues are determined by the values of
\begin{align}
T(y_{\rm imp},z_{\rm imp};\mu,\eta) &= {\rm trace} \left( U(y_{\rm imp},z_{\rm imp};\mu,\eta) \right), \\
D(y_{\rm imp},z_{\rm imp};\mu,\eta) &= \det \left( U(y_{\rm imp},z_{\rm imp};\mu,\eta) \right).
\label{eq:TD}
\end{align}
The following results provide asymptotic formulas for $T$ and $D$.
Due to the square-root singularity, we have $T \to \infty$ as $y_{\rm imp} \to 0$.
Thus in order to obtain a regular asymptotic expansion, we consider
the products $y_{\rm imp} T$ and $y_{\rm imp} D$.

We also require some additional notation.
In \eqref{eq:xi} we defined
\begin{equation}
\xi_p = \frac{\partial^2 P_{{\rm global},1}^p}{\partial z^2} \bigg|_{(x,z;\mu,\eta) = (0,z_{{\rm graz},0};0,0)},
\nonumber
\end{equation}
for all $p \ge 1$,
and we now also define
\begin{align}
d_p &= \frac{\partial^2 P_{{\rm global},1}^p}{\partial z \partial x} \bigg|_{(x,z;\mu,\eta) = (0,z_{{\rm graz},0};0,0)}, \nonumber \\
e_p &= \frac{\partial^2 P_{{\rm global},1}^p}{\partial z \partial \mu} \bigg|_{(x,z;\mu,\eta) = (0,z_{{\rm graz},0};0,0)}, \nonumber \\
f_p &= \frac{\partial^2 P_{{\rm global},1}^p}{\partial z \partial \eta} \bigg|_{(x,z;\mu,\eta) = (0,z_{{\rm graz},0};0,0)}. \nonumber
\end{align}
For brevity we write $\cO(\ell)$ to denote all terms in an expression that are order $\ell$ or greater.

\begin{lemma}
Consider a system \eqref{eq:f}--\eqref{eq:resetLaw} satisfying the assumptions of Theorem \ref{th:codim21}.
The products $y_{\rm imp} T(y_{\rm imp},z_{\rm imp};\mu,\eta)$
and $y_{\rm imp} D(y_{\rm imp},z_{\rm imp};\mu,\eta)$ are well-defined for $y_{\rm imp} > 0$
and admit $C^{k-2}$ extensions to a neighbourhood of $(y,z;\mu,\eta) = (0,z_{{\rm graz},0};0,0)$.
Moreover,
\begin{align}
y_{\rm imp} T(y_{\rm imp},z_{\rm imp};\mu,\eta) &=
\left( a_{11} \phi^2 + a_{22} - \frac{\alpha \xi_1 \omF^2 (1-\alpha)}{\gamma} \right) y_{\rm imp}
- \alpha \omF \xi_1 (z_{\rm imp} - z_{{\rm graz},0}) \nonumber \\
&\quad- \alpha \omF e_1 \mu - \alpha \omF a_{12}' \eta + \cO(2), \label{eq:yT1} \\
y_{\rm imp} D(y_{\rm imp},z_{\rm imp};\mu,\eta) &=
\phi^2 \delta y_{\rm imp} + \cO(2). \label{eq:yD1}
\end{align}
\label{le:traceDet1}
\end{lemma}

\begin{lemma}
Consider a system \eqref{eq:f}--\eqref{eq:resetLaw} satisfying the assumptions of Theorem \ref{th:codim2p}.
The products $y_{\rm imp} T(y_{\rm imp},z_{\rm imp};\mu,\eta)$
and $y_{\rm imp} D(y_{\rm imp},z_{\rm imp};\mu,\eta)$ are well-defined for $y_{\rm imp} > 0$
and admit $C^{k-2}$ extensions to a neighbourhood of $(y,z;\mu,\eta) = (0,z_{{\rm graz},0};0,0)$.
Moreover,
\begin{align}
y_{\rm imp} T(y_{\rm imp},z_{\rm imp};\mu,\eta) &=
-\left( \left( 1 + \phi^2 \right) \delta^{\frac{p}{2}} + \frac{\alpha \xi_p \omF^2 (1-\alpha)}{\gamma} \right) y_{\rm imp}
- \alpha \omF \xi_p (z_{\rm imp} - z_{{\rm graz},0}) \nonumber \\
&\quad- \alpha \omF e_p \mu - \alpha \omF f_p \eta + \cO(2), \label{eq:yTp} \\
y_{\rm imp} D(y_{\rm imp},z_{\rm imp};\mu,\eta) &=
\phi^2 \delta^p y_{\rm imp} + \cO(2), \label{eq:yDp}
\end{align}
where
\begin{equation}
f_p = \frac{a_{12} p \delta^{\frac{p}{2} - 1} \kappa_p'}{2 \sin^2 \left( \frac{\pi}{p} \right)}.
\label{eq:fp}
\end{equation}
\label{le:traceDetp}
\end{lemma}

\begin{proof}[Proof of Lemma \ref{le:traceDet1}]
The maps $P_{\rm virt}$ and $P_{\rm virt} \circ R$ are $C^k$ and given by \eqref{eq:xHatzHat} and \eqref{eq:x0z0}.
Thus their derivatives are $C^{k-1}$ and given by
\begin{align}
\rD P_{\rm virt}(y_{\rm imp},z_{\rm imp};\mu,\eta) &=
\begin{bmatrix}
\frac{\partial \hat{x}}{\partial y_{\rm imp}} & \frac{\partial \hat{x}}{\partial z_{\rm imp}} \\[1.4mm]
\frac{\partial \hat{z}}{\partial y_{\rm imp}} & \frac{\partial \hat{z}}{\partial z_{\rm imp}}
\end{bmatrix} =
\begin{bmatrix}
\frac{y_{\rm imp}}{\gamma} + \cO(2) & \cO(2) \\
\frac{\omF}{\gamma} + \cO(1) & 1 + \cO(1)
\end{bmatrix}, \label{eq:DPvirt} \\
\rD \left( P_{\rm virt} \circ R \right)(y_{\rm imp},z_{\rm imp};\mu,\eta) &=
\begin{bmatrix}
\frac{\partial x^{(0)}}{\partial y_{\rm imp}} & \frac{\partial x^{(0)}}{\partial z_{\rm imp}} \\[1.4mm]
\frac{\partial z^{(0)}}{\partial y_{\rm imp}} & \frac{\partial z^{(0)}}{\partial z_{\rm imp}}
\end{bmatrix} =
\begin{bmatrix}
\frac{\phi^2 y_{\rm imp}}{\gamma} + \cO(2) & \cO(2) \\
\frac{\omF(1-\alpha)}{\gamma} + \cO(1) & 1 + \cO(1)
\end{bmatrix}. \label{eq:DPvirtR}
\end{align}
We have $\det(\rD P_{\rm virt}) \to 0$ as $y_{\rm imp} \to 0$,
so there exists a $C^{k-2}$ function $G_1$ such that $\det(\rD P_{\rm virt}) = y_{\rm imp} G_1(y_{\rm imp},z_{\rm imp};\mu,\eta)$.
Moreover, $G_1(0,z_{{\rm graz},0};0,0) = \frac{1}{\gamma} \ne 0$ by \eqref{eq:DPvirt},
consequently $y_{\rm imp} \left( \rD P_{\rm virt} \right)^{-1}$ is $C^{k-2}$
(more formally it has a $C^{k-2}$ extension from points with $y_{\rm imp} > 0$,
to a neighbourhood of $(y,z;\mu,\eta) = (0,z_{{\rm graz},0};0,0)$).
Thus $y_{\rm imp} \rD P_{\rm disc} = y_{\rm imp} \rD \left( P_{\rm virt} \circ R \right) \left( \rD P_{\rm virt} \right)^{-1}$ is $C^{k-2}$,
and by \eqref{eq:DPvirt} and \eqref{eq:DPvirtR}
\begin{equation}
y_{\rm imp} \rD P_{\rm disc}(y_{\rm imp},z_{\rm imp};\mu,\eta) =
\begin{bmatrix}
\phi^2 y_{\rm imp} + \cO(2) & \cO(3) \\
-\alpha \omF + \cO(1) & y_{\rm imp} + \cO(2)
\end{bmatrix}.
\label{eq:DPdisc}
\end{equation}

Since $a_{12} = 0$ and $f_1 = a_{12}'$,
\begin{equation}
\rD P_{\rm global} \left( x^{(0)}, z^{(0)}; \mu, \eta \right) =
\begin{bmatrix}
a_{11} + \cO(1) & d_1 x^{(0)} + \xi_1 \left( z^{(0)} - z_{{\rm graz},0} \right) + e_1 \mu + a_{12}' \eta + \cO(2) \\
a_{21} + \cO(1) & a_{22} + \cO(1)
\end{bmatrix},
\label{eq:DPglobal}
\end{equation}
which is $C^{k-1}$.
By multiplying \eqref{eq:DPdisc} and \eqref{eq:DPglobal}, and inserting \eqref{eq:x0z0}, we obtain
\begin{equation}
y_{\rm imp} U(y_{\rm imp},z_{\rm imp};\mu,\eta) =
\begin{bmatrix}
X + \cO(2) & \cO(2) \\
-\alpha \omF a_{22} + \cO(1) & a_{22} y_{\rm imp} + \cO(2)
\end{bmatrix},
\label{eq:yU}
\end{equation}
where $X = a_{11} \phi^2 y_{\rm imp} - \alpha \omF \left( \xi_1 \left( \frac{\omF(1-\alpha)}{\gamma} \,y_{\rm imp} + z_{\rm imp} - z_{{\rm graz},0} \right)
+ e_1 \mu + a_{12}' \eta \right)$.
By evaluating the trace of \eqref{eq:yU} we obtain \eqref{eq:yT1}.
Also using \eqref{eq:DPdisc} and \eqref{eq:DPglobal} we obtain \eqref{eq:yD1}.
\end{proof}

\begin{proof}[Proof of Lemma \ref{le:traceDetp}]
By assumption $\kappa_p = 0$, i.e.~$\tau = g_p(\delta)$.
Thus $S_p = 0$ by Lemma \ref{le:SpTp}(a),
and $S_{p+1} = -\delta^{\frac{p}{2}}$ by Lemma \ref{le:SpTp}(b).
So from \eqref{eq:Ap},
\begin{equation}
\rD P_{\rm global}^p \left( x^{(0)}, z^{(0)}; \mu, \eta \right) =
\begin{bmatrix}
-\delta^{\frac{p}{2}} + \cO(1) & d_p x^{(0)} + \xi_p \left( z^{(0)} - z_{{\rm graz},0} \right) + e_p \mu + f_p \eta + \cO(2) \\
\cO(1) & -\delta^{\frac{p}{2}} + \cO(1)
\end{bmatrix},
\label{eq:DPglobalp}
\end{equation}
which is $C^{k-1}$ because $P_{\rm global}$ is $C^k$.
Also from \eqref{eq:Ap}, $f_p = a_{12} \frac{\partial S_p}{\partial \eta}$,
and it is a simple calculus exercise to verify \eqref{eq:fp} by differentiating \eqref{eq:Sp3}.
Following the previous proof, we then multiply \eqref{eq:DPdisc} and \eqref{eq:DPglobalp} with \eqref{eq:x0z0},
and take the trace and determine to arrive at \eqref{eq:yTp} and \eqref{eq:yDp}.
\end{proof}

\subsection{Main arguments for resonant grazing bifurcations}

Given $p \ge 1$, a $p$-loop MPS corresponds to a zero of the VIVID function $V$.
At the codimension-two point $(\mu,\eta) = (0,0)$,
the Jacobian matrix $J$ of $V$, see \eqref{eq:Jac1Jac2},
is singular, see \eqref{eq:detJac1detJac2}.
Thus we cannot directly solve $V(y,z;\mu,\eta) = (0,0)$ for $y$ and $z$.
However, the alternate matrix $K$, see again \eqref{eq:Jac1Jac2},
is non-singular, see again \eqref{eq:detJac1detJac2}.
Thus we can solve $V(y,z;\mu,\eta) = (0,0)$ for $z$ and $\mu$, at least locally.

\begin{proof}[Proof of Theorem \ref{th:codim21}]
Since $\beta \ne 0$, we have $\det(K) \ne 0$ by \eqref{eq:detJac1detJac2}.
Thus, by the implicit function theorem, there exist unique $C^k$ functions $Z$ and $M$
such that $V(y,Z(y,\eta);M(y,\eta),\eta) = (0,0)$ for all $(y,\eta)$ in a neighbourhood of $(0,0)$.
In view of Assumption \ref{ass:grazing2}, $Z(0,\eta) = z_{{\rm graz},\eta}$ and $M(0,\eta) = 0$
for all sufficiently small $\eta \in \mathbb{R}$.
Also $M(y,\eta)$ is second-order because $\det(J) = 0$.
Thus
\begin{equation}
\begin{split}
Z(y,\eta) = z_{{\rm graz},\eta} + k_1 y + \cO \left( \left( |y| + |\eta| \right)^2 \right), \\
M(y,\eta) = k_2 y^2 + k_3 \eta y + \cO \left( \left( |y| + |\eta| \right)^3 \right),
\end{split}
\label{eq:ZM}
\end{equation}
for some constants $k_1, k_2, k_3 \in \mathbb{R}$.

With $p=1$, the VIVID function is $V(y,z;\mu,\eta) = \left( x^{(1)}, z^{(1)} \right) - \left( \hat{x}, \hat{z} \right)$.
By \eqref{eq:Pglobal} and \eqref{eq:DPglobal}, we have
\begin{equation}
\begin{split}
x^{(1)} &= a_{11} x^{(0)} + b_1 \mu + \tilde{\cO}(2)
+ \left( d_1 x^{(0)} + e_1 \mu + a_{12}' \eta + \tilde{\cO}(2) \right) \left( z^{(0)} - z_{{\rm graz},\eta} \right) \\
&\quad+ \left( \frac{\xi_1}{2} + \tilde{\cO}(1) \right) \left( z^{(0)} - z_{{\rm graz},\eta} \right)^2
+ \cO \left( \left( z^{(0)} - z_{{\rm graz},\eta} \right)^3 \right), \\
z^{(1)} &= z_{{\rm graz},\eta} + a_{21} x^{(0)} + a_{22} \left( z^{(0)} - z_{{\rm graz},\eta} \right) + b_2 \mu \\
&\quad+ \cO \left( \left( \left| x^{(0)} \right| + \left| z^{(0)} - z_{{\rm graz},\eta} \right| + |\mu| + |\eta| \right)^2 \right),
\end{split}
\label{eq:codim21proof1}
\end{equation}
using the abbreviation $\tilde{\cO}(\ell) = \cO \left( \left( \left| x^{(0)} \right| + |\mu| + |\eta| \right)^\ell \right)$ for $\ell = 1,2$.
The way in which we have grouped terms in \eqref{eq:codim21proof1}
is helpful because in the desired result $x^{(0)}$ and $\mu$
are higher order than $z^{(0)} - z_{{\rm graz},\eta}$ and $\eta$.
By substituting \eqref{eq:x0z0} and \eqref{eq:ZM} into \eqref{eq:codim21proof1}, and then subtracting \eqref{eq:xHatzHat}, we obtain
\begin{equation}
V(y,Z(y,\eta);M(y,\eta),\eta) =
\begin{bmatrix}
X_2 y^2 + X_3 \eta y + \cO \left( \left( |y| + |\eta| \right)^3 \right) \\
X_1 y + \cO \left( \left( |y| + |\eta| \right)^2 \right)
\end{bmatrix},
\nonumber
\end{equation}
where
\begin{align}
X_1 &= -(1-a_{22}) k_1 + \frac{\omF}{\gamma} \left( a_{22}(1-\alpha) - 1 \right), \nonumber \\
X_2 &= b_1 k_2 + \frac{1}{2 \gamma} \left( a_{11} \phi^2 - 1 \right)
+ \frac{\xi_1}{2} \left( k_1 + \frac{(1-\alpha) \omF}{\gamma} \right)^2, \nonumber \\
X_3 &= b_1 k_3 + a_{12}' \left( k_1 + \frac{(1-\alpha) \omF}{\gamma} \right). \nonumber
\end{align}
To have $V(y,Z(y,\eta);M(y,\eta),\eta) = (0,0)$, we require $X_i = 0$ for all $i = 1,2,3$, thus
\begin{align}
k_1 &= \frac{\omF \left( a_{22}(1-\alpha) - 1 \right)}{(1-a_{22}) \gamma}, &
k_2 &= \frac{-s_1^+}{2 (1 - a_{22}) b_1 \gamma}, &
k_3 &= \frac{\alpha \omF a_{12}'}{(1-a_{22}) b_1 \gamma},
\label{eq:k1231}
\end{align}
using the formula \eqref{eq:c1} for $s_1^+$.

Now define
\begin{equation}
\Psi^\pm(y,\eta) = T(y,Z(y,\eta);M(y,\eta),\eta) \mp D(y,Z(y,\eta);M(y,\eta),\eta) \mp 1.
\label{eq:Psi}
\end{equation}
The one-loop MPS has a stability multiplier of $1$ when $\Psi^+(y,\eta) = 0$,
and a stability multiplier of $-1$ when $\Psi^-(y,\eta) = 0$.
By Lemma \ref{le:traceDet1}, the product $y \Psi^\pm(y,\eta)$ is $C^{k-2}$.
By inserting \eqref{eq:ZM} with \eqref{eq:k1231} into \eqref{eq:yT1} and \eqref{eq:yD1}
we obtain (after some algebraic manipulation),
\begin{equation}
y \Psi^\pm(y,\eta) = s^\pm_1 y - \alpha \omF a_{12}' \eta + \cO \left( \left( |y| + |\eta| \right)^2 \right),
\nonumber
\end{equation}
using again \eqref{eq:c1}.
Since $s^\pm_1 \ne 0$, by the implicit function theorem
there exists a unique $C^{k-2}$ function $\Upsilon^\pm(\eta)$ such that
$\Upsilon^\pm(\eta) \Psi^\pm \left( \Upsilon^\pm(\eta), \eta \right) = 0$
for all $\eta$ in a neighbourhood of $0$, and
\begin{equation}
\Upsilon^\pm(\eta) = \frac{\alpha \omF a_{12}'}{s^\pm_1} \,\eta + \cO \left( \eta^2 \right).
\label{eq:Upsilon1}
\end{equation}
Let $g_{{\rm SN},1}(\eta) = M \left( \Upsilon^+(\eta), \eta \right)$
and $g_{{\rm PD},1}(\eta) = M \left( \Upsilon^+(\eta), \eta \right)$.
These functions are $C^{k-2}$, and from \eqref{eq:k1231} and \eqref{eq:Upsilon1}
we obtain \eqref{eq:coeffs1} after simplification.
Notice $k_2 \ne 0$ by \eqref{eq:k1231},
so the coefficient of the $y^2$-term in \eqref{eq:ZM} is non-zero.
Thus as we move away from $\mu = g_{{\rm SN},1}(\eta)$
by perturbing the value of $\mu$ in the appropriate direction,
two zeros of the VIVID function separate from one another at a rate
asymptotically proportional to the square root of the change in $\mu$.
Therefore $\mu = g_{{\rm SN},1}(\eta)$ is a curve of saddle-node bifurcations.
With ${\rm sgn}(\eta) = {\rm sgn} \left( a_{12}' s_1^\pm \right)$,
the impact velocity $\Upsilon^\pm(\eta)$ of the one-loop MPS is positive
for sufficiently small $\eta \ne 0$ by \eqref{eq:Upsilon1},
and because $\alpha > 0$ and $\omF > 0$.
Thus for this sign of $\eta$ the one-loop MPS is admissible.
\end{proof}

\begin{proof}[Proof of Theorem \ref{th:codim2p}]
As in the previous proof there exist unique $C^k$ functions $Z$ and $M$ of the form \eqref{eq:ZM}
such that $V(y,Z(y,\eta);M(y,\eta),\eta) = (0,0)$ for all $(y,\eta)$ in a neighbourhood of $(0,0)$.
By \eqref{eq:Pglobalp}--\eqref{eq:b1pb2p} and \eqref{eq:DPglobalp},
\begin{equation}
\begin{split}
x^{(p)} &= -\delta^{\frac{p}{2}} x^{(0)} + T_p \beta \mu + \tilde{\cO}(2)
+ \left( d_p x^{(0)} + e_p \mu + f_p \eta + \tilde{\cO}(2) \right) \left( z^{(0)} - z_{{\rm graz},\eta} \right) \\
&\quad+ \left( \frac{\xi_p}{2} + \tilde{\cO}(1) \right) \left( z^{(0)} - z_{{\rm graz},\eta} \right)^2
+ \cO \left( \left( z^{(0)} - z_{{\rm graz},\eta} \right)^3 \right), \\
z^{(p)} &= z_{{\rm graz},\eta} - \delta^{\frac{p}{2}} \left( z^{(0)} - z_{{\rm graz},\eta} \right) + 
T_p \left( a_{21} b_1 + (1 - a_{11}) b_2 \right) \mu \\
&\quad+
\cO \left( \left( \left| x^{(0)} \right| + \left| z^{(0)} - z_{{\rm graz},\eta} \right| + |\mu| + |\eta| \right)^2 \right).
\end{split}
\label{eq:codim21proofp}
\end{equation}
By substituting \eqref{eq:x0z0} and \eqref{eq:ZM} into \eqref{eq:codim21proofp}, and then subtracting \eqref{eq:xHatzHat}, we obtain
\begin{equation}
V(y,Z(y,\eta);M(y,\eta),\eta) =
\begin{bmatrix}
X_2 y^2 + X_3 \eta y + \cO \left( \left( |y| + |\eta| \right)^3 \right) \\
X_1 y + \cO \left( \left( |y| + |\eta| \right)^2 \right)
\end{bmatrix},
\nonumber
\end{equation}
where
\begin{align}
X_1 &= -\left( 1 + \delta^{\frac{p}{2}} \right) k_1 - \frac{\omF}{\gamma} \left( 1 + \delta^{\frac{p}{2}} (1-\alpha) \right), \nonumber \\
X_2 &= T_p \beta k_2 - \frac{1}{2 \gamma} \left( 1 + \delta^{\frac{p}{2}} \phi^2 \right)
+ \frac{\xi_p}{2} \left( k_1 + \frac{(1-\alpha) \omF}{\gamma} \right)^2, \nonumber \\
X_3 &= T_p \beta k_3 + f_p \left( k_1 + \frac{(1-\alpha) \omF}{\gamma} \right). \nonumber
\end{align}
Solving $X_i = 0$ for all $i = 1,2,3$, gives
\begin{align}
k_1 &= -\frac{\omF \left( 1 + \delta^{\frac{p}{2}} (1-\alpha) \right)}
{\left( 1 + \delta^{\frac{p}{2}} \right) \gamma}, &
k_2 &= \frac{-s_p^+}{2 \left( 1 + \delta^{\frac{p}{2}} \right) \beta \gamma T_p}, &
k_3 &= \frac{\alpha \omF f_p}{\left( 1 + \delta^{\frac{p}{2}} \right) \beta \gamma T_p},
\label{eq:k123p}
\end{align}
using \eqref{eq:cp}.
By inserting \eqref{eq:ZM} with \eqref{eq:k123p} into \eqref{eq:yTp} and \eqref{eq:yDp}, we obtain
\begin{equation}
y \Psi^\pm(y,\eta) = s^\pm_p y - \alpha \omF f_p \eta + \cO \left( \left( |y| + |\eta| \right)^2 \right),
\nonumber
\end{equation}
using again \eqref{eq:cp}.
Since $s^\pm_p \ne 0$, by the implicit function theorem
there exists a unique $C^{k-2}$ function $\Upsilon^\pm(\eta)$ such that
$\Upsilon^\pm(\eta) \Psi^\pm \left( \Upsilon^\pm(\eta), \eta \right) = 0$
for all $\eta$ in a neighbourhood of $0$, and
\begin{equation}
\Upsilon^\pm(\eta) = \frac{\alpha \omF f_p}{s^\pm_p} \,\eta + \cO \left( \eta^2 \right).
\label{eq:Upsilonp}
\end{equation}
Let $g_{{\rm SN},p}(\eta) = M \left( \Upsilon^+(\eta), \eta \right)$
and $g_{{\rm PD},p}(\eta) = M \left( \Upsilon^+(\eta), \eta \right)$.
These functions are $C^{k-2}$, and from \eqref{eq:k123p} and \eqref{eq:Upsilonp}
we obtain \eqref{eq:coeffsp},
using also Lemma \ref{le:SpTp} and \eqref{eq:fp} for $T_p$ and $f_p$.
As in the previous proof, $\mu = g_{{\rm SN},p}(\eta)$ is a curve
of saddle-node bifurcations because the coefficient of the $y^2$-term in \eqref{eq:ZM}
is non-zero by \eqref{eq:k123p}.

Finally, we show that
on the bifurcation curves $\mu = g_{{\rm SN},p}(\eta)$ and $\mu = g_{{\rm PD},p}(\eta)$,
the $p$-loop MPS is admissible for sufficiently small $\eta \ne 0$ with ${\rm sgn}(\eta) = {\rm sgn} \left( \kappa_p' s_p^\pm \right)$.
Observe ${\rm sgn}(f_p) = {\rm sgn} \left( a_{12} \kappa_p' \right)$, by \eqref{eq:fp}.
So by \eqref{eq:Upsilonp},
${\rm sgn} \left( \Upsilon^\pm(\eta) \right) = {\rm sgn} \left( \alpha a_{12} \kappa_p' s_p^\pm \eta \right)$,
for sufficiently small $\eta \ne 0$.
Thus ${\rm sgn}(\eta) = {\rm sgn} \left( \kappa_p' s_p^\pm \right)$
ensures that the impact velocity $y^*(\eta) = \Upsilon^\pm(\eta)$ of the $p$-loop MPS
is positive for sufficiently small $\eta \ne 0$.

On the bifurcation curve, the $z$-value of the $p$-loop MPS at impact is
\begin{align}
z^*(\eta) = Z \left( \Upsilon^\pm(\eta), \eta \right)
= z_{{\rm graz},\eta} + \frac{\alpha \omF f_p k_1}{s^\pm_p} \,\eta + \cO \left( \eta^2 \right),
\label{eq:admp7}
\end{align}
by \eqref{eq:ZM} and \eqref{eq:Upsilonp}.
Analogous to the proof of Theorem \ref{th:codim1p},
let $\left( x^{(0)}(\eta), z^{(0)}(\eta) \right) = P_{\rm virt} \left( R \left( y^*(\eta), z^*(\eta) \right) \right)$
and $\left( x^{(j)}(\eta), z^{(j)}(\eta) \right) = P_{\rm global} \left( x^{(j-1)}, z^{(j-1)} \right)$ for all $j = 1,2,\ldots,p$.
Trivially, $x^{(0)}(\eta) > 0$ and $x^{(p)}(\eta) > 0$, by \eqref{eq:xHatzHat} and \eqref{eq:x0z0},
so it remains to show $x^{(j)}(\eta) < 0$ for all $j = 1,2,\ldots,p-1$ and sufficiently small $\eta \ne 0$
with ${\rm sgn}(\eta) = {\rm sgn} \left( \kappa_p' s_p^\pm \right)$.

By \eqref{eq:x0z0} and \eqref{eq:admp7},
\begin{align}
x^{(0)}(\eta) &= \cO \left( \eta^2 \right), \label{eq:admp8a} \\
z^{(0)}(\eta) &= z_{{\rm graz},\eta} + \frac{\alpha \omF f_p}{s^\pm_p} \left( k_1 + \tfrac{\omF(1-\alpha)}{\gamma} \right) \eta + \cO \left( \eta^2 \right) \nonumber \\
&= z_{{\rm graz},\eta} - \frac{\alpha^2 \omF^2 f_p}{\gamma \left( 1 + \delta^{\frac{p}{2}} \right) s^\pm_p} \,\eta + \cO \left( \eta^2 \right), \label{eq:admp8b}
\end{align}
using also \eqref{eq:k123p}.
On the bifurcation curve, $\mu = M \left( \Upsilon^\pm(\eta), \eta \right) = \cO \left( \eta^2 \right)$,
so the constant term in $P_{\rm global}$ is second-order, and hence to first order we have simply
\begin{equation}
\begin{bmatrix} x^{(j)}(\eta) \\ z^{(j)}(\eta) \end{bmatrix} =
\begin{bmatrix} 0 \\ z_{{\rm graz},\eta} \end{bmatrix} + A^j
\begin{bmatrix} x^{(0)}(\eta) \\ z^{(0)}(\eta) - z_{{\rm graz},\eta} \end{bmatrix} + \cO \left( \eta^2 \right).
\label{eq:admp9}
\end{equation}
By \eqref{eq:Ap}, the $(1,2)$-entry of $A^j$ is $a_{12} S_j$, so by substituting \eqref{eq:admp8a} and \eqref{eq:admp8b} into \eqref{eq:admp9} we obtain
\begin{equation}
x^{(j)}(\eta) = -\frac{a_{12} S_j \alpha^2 \omF^2 f_p}{\gamma \left( 1 + \delta^{\frac{p}{2}} \right) s^\pm_p} \eta + \cO \left( \eta^2 \right).
\label{eq:admp10}
\end{equation}
By \eqref{eq:fp}, ${\rm sgn}(f_p) = {\rm sgn} \left( a_{12} \kappa_p' \right)$.
Also $S_j > 0$, for all $j = 1,2,\ldots,p-1$, because at resonance $\theta = \frac{\pi}{p}$ in \eqref{eq:Sp3}.
So \eqref{eq:admp10} implies
\begin{equation}
{\rm sgn} \left( x^{(j)}(\eta) \right) = -{\rm sgn} \left( \kappa_p' s^\pm_p \eta \right),
\label{eq:admp11}
\end{equation}
for sufficiently small $\eta \ne 0$.
Hence with ${\rm sgn}(\eta) = {\rm sgn} \left( \kappa_p' s_p^\pm \right)$,
we have $x^{(j)}(\eta) < 0$ for all $j = 1,2,\ldots,p-1$ as required.
\end{proof}

\section{Discussion}
\label{sec:conc}


Grazing bifurcations elicit a severe nonlinearity, so often induce a chaotic response.
Yet, the response can be periodic,
due to the presence of resonant grazing bifurcations.
We have performed the first formal unfolding
of resonant grazing bifurcations to explain how these initiate periodicity regions
bounded by curves of saddle-node and period-doubling bifurcations.
We have obtained explicit formulas for the coefficients
of the leading-order (quadratic) terms of these curves.
The formulas are valuable as they reveal the direction and shape
of the curves as they emanate from the resonant grazing bifurcation.
For example, for the linear impact oscillator model,
the curves emanate in different directions for different $p=1$ resonant grazing bifurcations.
Theorem \ref{th:codim21} shows how the direction is governed by the sign of $a_{12}'$,
and Proposition \ref{pr:imp1} shows how the sign of $a_{12}'$ is determined by the parity of $n$.
This explains why in Fig.~\ref{fig:G} the geometry of the $p=1$ periodicity region 
with resonant point at $\omega \approx 0.5$
is starkly different to the geometries of the other periodicity regions.
This region corresponds to $n=4$, while the other $p=1$ regions correspond to $n=3$ and $n=5$.

The theorems in \S\ref{sec:results}
apply to any impact oscillator model that
can be put in the form \eqref{eq:f}--\eqref{eq:resetLaw}.
For a nonlinear oscillator, the quantities in the theorems
can be evaluated by simulating the Poincar\'e map $P_{\rm global}$ numerically,
and evaluating its derivatives through finite difference approximations.
The theorems can also be applied to grazing bifurcations of periodic solutions
that have impacts at other parts of their trajectories,
as these impacts do not affect the smoothness of $P_{\rm global}$.

It remains to determine if there is a simple criterion
that dictates the criticality of the period-doubling bifurcations.
All of the period-doubling bifurcations that we have identified for the linear impact oscillator model
appear to be subcritical, i.e.~create an unstable period-doubled solution.
Each period-doubled solution undergoes grazing along another bifurcation curve
emanating from the resonant grazing bifurcation.
We have not computed these curves as they do not bound the periodicity regions,
and an asymptotic computation of these curves is beyond the scope of this work.
The inclusion of these curves
would bring the unfolding closer to that of a {\em homoclinic-doubling bifurcation}:
codimension-two points on curves of homoclinic bifurcations from which there issue curves
of saddle-node, period-doubling, and homoclinic bifurcations of
the period-doubled solution \cite{HoSa10,HoKo01,OlKr00}.
It is not yet clear why resonant grazing and homoclinic-doubling bifurcations admit similar unfoldings.
Homoclinic-doubling bifurcations often appear in cascades,
which is not what we see for resonant grazing bifurcations, but is
reminiscent of cascades of grazing-sliding bifurcations.
These result from boundary equilibrium bifurcations
with homoclinic or heteroclinic connections \cite[pg.~373]{DiBu08},
and yet to be fully explained.

\section*{Acknowledgements}

This work was supported by Marsden Fund contract MAU2209 managed by Royal Society Te Ap\={a}rangi.
The authors thank John Bailie and Soumitro Banerjee for helpful conversations.

\appendix

\section{Additional calculations regarding the eigenvalues of $A$}
\label{app:someProofs}

By definition,
\begin{equation}
H_p(\tau,\delta) = \sum_{j=1}^{p-1} \sum_{k=1}^j \lambda_1^{k-j} \lambda_2^{1-k},
\label{eq:HpAgain}
\end{equation}
repeating \eqref{eq:Hp}.
Since $0 < \delta < 1$, we have $\lambda_1, \lambda_2 \ne 0$, $\lambda_2 \ne 1$, and
\begin{align}
\lambda_1 &= 1 \text{~if and only if $\tau = \delta + 1$, and} \nonumber  \\
\lambda_1 &= \lambda_2 \text{~if and only if $\tau = 2 \sqrt{\delta}$.} \nonumber 
\end{align}
If $\tau \ne \delta + 1$ and $\tau \ne 2 \sqrt{\delta}$,
then by three applications of the formula
for the sum of a truncated geometric series,
\begin{align}
H_p(\tau,\delta) &= \frac{\lambda_1 \lambda_2}{\lambda_1 - \lambda_2}
\sum_{j=1}^{p-1} \left( \lambda_2^{-j} - \lambda_1^{-j} \right) \label{eq:Hp2} \\
&= \frac{\lambda_1 \lambda_2}{\lambda_1 - \lambda_2}
\left( \frac{\lambda_2^{p-1} - 1}{\lambda_2^{p-1} \left( \lambda_2 - 1 \right)}
- \frac{\lambda_1^{p-1} - 1}{\lambda_1^{p-1} \left( \lambda_1 - 1 \right)} \right) \nonumber \\
&= \frac{\lambda_1^{p-1} \lambda_2^{p-1} (\lambda_1 - \lambda_2) - \left( \lambda_1^p - \lambda_2^p \right)
+ \left( \lambda_1^{p-1} - \lambda_2^{p-1} \right)}
{(1 - \lambda_1)(1 - \lambda_2)(\lambda_1 - \lambda_2) \lambda_1^{p-2} \lambda_2^{p-2}}. \label{eq:Hp3}
\end{align}
In a similar fashion we obtain from \eqref{eq:SpTp},
\begin{align}
T_p &= \frac{1}{\lambda_1 - \lambda_2} \sum_{j=1}^{p-1} \left( \lambda_1^j - \lambda_2^j \right) \label{eq:Tp2} \\
&= \frac{\lambda_1 - \lambda_2 - \left( \lambda_1^p - \lambda_2^p \right) + \lambda_1 \lambda_2 \left( \lambda_1^{p-1} - \lambda_2^{p-1} \right)}
{(1 - \lambda_1)(1 - \lambda_2)(\lambda_1 - \lambda_2)}. \label{eq:Tp3}
\end{align}

\begin{proof}[Proof of Lemma \ref{le:admissibilityBoundary}]
If $\tau \ge 2 \sqrt{\delta}$,
then $\lambda_1$ and $\lambda_2$ are real and positive,
so by \eqref{eq:HpAgain}, $H_p(\tau,\delta) > 0$ for all $\tau \ge 2 \sqrt{\delta}$.

Thus for the remainder of the proof it suffices to consider $g_{\frac{p}{2}}(\delta) \le \tau < 2 \sqrt{\delta}$,
with which $\lambda_1$ and $\lambda_2$ are complex-valued.
In this case
\begin{align}
\lambda_1 &= r \re^{\ri \theta}, &
\lambda_2 &= r \re^{-\ri \theta},
\label{eq:polarCoords}
\end{align}
where $r = \sqrt{\delta}$ and 
$\theta = \cos^{-1} \left( \frac{\tau}{2 \sqrt{\delta}} \right) \in (0,\pi)$.
For any $k \in \mathbb{Z}$, we have
\begin{equation}
\lambda_1^k - \lambda_2^k = 2 \ri r^k \sin(k \theta),
\label{eq:lam1kminuslam2k}
\end{equation}
thus by \eqref{eq:Hp2}
\begin{equation}
H_p(\tau,\delta) = \frac{1}{\sin(\theta)} \sum_{j=1}^{p-1} r^{1-j} \sin(j \theta).
\label{eq:Hp4}
\end{equation}
Thus $H_p(\tau,\delta) > 0$ for all $\theta \in \left( 0, \frac{\pi}{p-1} \right]$,
which is equivalent to $\tau \in \left[ g_{p-1}(\delta), 2 \sqrt{\delta} \right)$.

By instead inserting \eqref{eq:lam1kminuslam2k} into \eqref{eq:Hp3}, we obtain
\begin{equation}
H_p(\tau,\delta) = \frac{r^p \sin(\theta) - r \sin(p \theta) + \sin((p-1) \theta)}
{(\delta - \tau + 1) r^{p-2} \sin(\theta)}.
\label{eq:Hp5}
\end{equation}
By substituting $\theta = \frac{2 \pi}{p}$ into \eqref{eq:Hp5} (possible because $p \ge 3$),
we obtain (after simplification)
\begin{equation}
H_p \left( g_{\frac{p}{2}}(\delta), \delta \right) = \frac{r^p - 1}{(\delta - \tau + 1) r^{p-2}} < 0,
\nonumber
\end{equation}
using also $r < 1$.
In summary, we have shown $H_p \left( g_{\frac{p}{2}}(\delta), \delta \right) < 0$
and $H_p(\tau,\delta) > 0$, for all $\tau \ge g_{p-1}(\delta)$.
Thus, by the intermediate value theorem, there exists
$\tau^* \in \left( g_{\frac{p}{2}}(\delta), g_{p-1}(\delta) \right)$
such that $H_p(\tau^*,\delta) = 0$,
and $h_p(\delta)$ is the largest such $\tau^*$.
\end{proof}

\begin{proof}[Proof of Lemma \ref{le:SpTp}(a)]
If $\tau \ge 2 \sqrt{\delta}$, then $\lambda_1$ and $\lambda_2$ are real and positive,
and by \eqref{eq:SpTp} we have $S_p > 0$.
Thus it suffices to assume $\tau < 2 \sqrt{\delta}$,
with which the eigenvalues can be written in the form \eqref{eq:polarCoords}.
Also $\lambda_1 \ne \lambda_2$, so
\begin{equation}
S_p = \frac{\lambda_1^p - \lambda_2^p}{\lambda_1 - \lambda_2}.
\label{eq:Sp2}
\end{equation}
For any $k \in \mathbb{Z}$, $\lambda_1^k - \lambda_2^k = 2 \ri r^k \sin(k \theta)$, thus
\begin{equation}
S_p = \frac{r^{p-1} \sin(p \theta)}{\sin(\theta)}.
\label{eq:Sp3}
\end{equation}
Thus $S_p = 0$ if and only if $\theta = \frac{m \pi}{p}$ for some $m = 1,2,\ldots,p-1$;
equivalently $\tau = g_{\frac{p}{m}}(\delta)$ for some $m = 1,2,\ldots,p-1$.
But we cannot have $\tau = g_{\frac{p}{m}}(\delta)$ for any $m \ge 2$ because
$\tau > h_p(\delta)$ and $h_p(\delta) > g_{\frac{p}{2}}(\delta)$ by Lemma \ref{le:admissibilityBoundary}.
\end{proof}

\begin{proof}[Proof of Lemma \ref{le:SpTp}(b)]
By \eqref{eq:Sp3}, $S_{p+1} = \frac{r^p \sin((p+1) \theta)}{\sin(\theta)}$,
where $\sin((p+1) \theta) = \sin(p \theta) \cos(\theta) + \cos(p \theta) \sin(\theta)$.
With $\tau = g_p(\delta)$, we have $\theta = \frac{\pi}{p}$ by the definition \eqref{eq:gp} of $g_p$,
so $\sin(p \theta) = 0$ and $\cos(p \theta) = -1$.
Thus $S_{p+1} = -r^p = -\delta^{\frac{p}{2}}$.

By \eqref{eq:Tp3} and \eqref{eq:lam1kminuslam2k}
\begin{equation}
T_p = \frac{\sin(\theta) - r^{p-1} \sin(p \theta) + r^p \sin((p-1) \theta)}{(\delta - \tau + 1) \sin(\theta)}.
\label{eq:Tp6}
\end{equation}
But $\tau = g_p(\delta)$, so
$\sin(p \theta) = 0$ and $\cos(p \theta) = -1$, and hence
$\sin((p-1) \theta) = \sin(p \theta) \cos(\theta) - \cos(p \theta) \sin(\theta) = \sin(\theta)$.
Thus $T_p = \frac{1 + \delta{\frac{p}{2}}}{\delta - \tau + 1}$
because $r = \sqrt{\delta}$.
\end{proof}

\begin{proof}[Proof of Lemma \ref{le:SpTp}(c)]
Since $T_2 = 1$, we can assume $p \ge 3$.
If $\tau \ge 2 \sqrt{\delta}$, then $\lambda_1$ and $\lambda_2$
are real and positive, so $T_p > 0$ by \eqref{eq:SpTp}.

So it remains to consider $g_{\frac{p}{2}}(\delta) \le \tau < 2 \sqrt{\delta}$ and use
the polar form \eqref{eq:polarCoords}.
By \eqref{eq:Tp2} and \eqref{eq:lam1kminuslam2k},
\begin{equation}
T_p = \frac{1}{\sin(\theta)} \sum_{j=1}^{p-1} r^{j-1} \sin(j \theta).
\label{eq:Tp4}
\end{equation}
With $0 < \theta < \frac{\pi}{p-1}$ (equivalently $g_{p-1}(\delta) < \tau < 2 \sqrt{\delta}$),
each term in \eqref{eq:Tp4} is positive, hence $T_p > 0$.

So it remains to consider $\frac{\pi}{p-1} \le \theta \le \frac{2 \pi}{p}$.
The formula \eqref{eq:Tp6} can be rewritten as
\begin{equation}
T_p = \frac{r^{p-1} \Xi(\theta) - r^{p-1} (1-r) \sin((p-1) \theta)
+ \left( 1 - r^{p-1} \right) \sin(\theta)}{(\delta - \tau + 1) \sin(\theta)},
\label{eq:Tp5}
\end{equation}
where
\begin{equation}
\Xi(\theta) = \sin(\theta) - \sin(p \theta) + \sin((p-1) \theta).
\nonumber
\end{equation}
By various trigonometric identities, this can be factored as
\begin{equation}
\Xi(\theta) = 4 \sin \left( \tfrac{\theta}{2} \right)
\sin \left( \tfrac{(p-1) \theta}{2} \right)
\sin \left( \tfrac{p \theta}{2} \right).
\nonumber
\end{equation}
Thus $\theta \le \frac{2 \pi}{p}$ implies $\Xi(\theta) \ge 0$,
so the first term in the numerator of \eqref{eq:Tp5} is greater than or equal to zero.
The second term in the numerator is also greater than or equal to zero
because $\frac{\pi}{p-1} \le \theta \le \frac{2 \pi}{p}$.
The third term in the numerator is strictly greater than zero,
thus $T_p > 0$.
\end{proof}

\begin{lemma}
Let $0 < \delta <1$, $p \ge 1$, and $0 < j < p$.
Then 
\begin{equation}
j \left( 1 - \delta^p \right) - p \left( 1 - \delta^j \right) < 0.
\label{eq:ujNumerator}
\end{equation}
\label{le:ujNumerator}
\end{lemma}

\begin{proof}
The line tangent to $f(x) = x^{\frac{j}{p}}$ at $x = 1$
is $f_{\rm tang}(x) = 1 + \frac{j}{p} (x - 1)$.
Since $f(x)$ is concave down,
$f(x) < f_{\rm tang}(x)$ for all $0 < x < 1$.
Substituting $x = \delta^p$ into $f(x) < f_{\rm tang}(x)$ gives
\begin{equation}
\delta^j < 1 + \frac{j}{p} \left( \delta^p - 1 \right),
\nonumber
\end{equation}
which is equivalent to \eqref{eq:ujNumerator}.
\end{proof}

\section{Calculations for the linear impact oscillator}
\label{app:linearImpactOsc}

\begin{proof}[Proof of Proposition \ref{pr:oscAb}]
Here we derive \eqref{eq:oscA} and \eqref{eq:oscb}.
Write $(x',z') = P_{\rm global}(x,z;\cA-\cA_{\rm graz}(\omF))$.
This map corresponds to an orbit in $(x,y)$-phase space
that starts at $(x,0)$ at time $\frac{z}{\omF}$,
and ends at $(x',0)$ at time $\frac{z' + 2 \pi}{\omF}$.
The orbit is given explicitly as $(\phi,\dot{\phi})$, so
\begin{equation}
\begin{split}
x' &= \phi \left( \tfrac{z' + 2 \pi}{\omF}; x, 0, \tfrac{z}{\omF}; \cA \right), \\
0 &= \dot{\phi} \left( \tfrac{z' + 2 \pi}{\omF}; x, 0, \tfrac{z}{\omF}; \cA \right).
\end{split}
\label{eq:Abproof1}
\end{equation}
At grazing, i.e.~with $(x,z;\cA) = (0,z_{\rm graz};\cA_{\rm graz}(\omF))$, we map to $(x',z') = (0,z_{\rm graz})$.
Given suitably small $\delta_1, \delta_2, \delta_3 \in \mathbb{R}$,
with the perturbed values
$(x,z;\cA) = \left( \delta_1, z_{\rm graz} + \delta_2; \cA_{\rm graz}(\omF) + \delta_3 \right)$, we map to
$(x',z') = \left( a_{11} \delta_1 + a_{12} \delta_2 + b_1 \delta_3 + \cO(2),
z_{\rm graz} + a_{21} \delta_1 + a_{22} \delta_2 + b_2 \delta_3 + \cO(2) \right)$,
where $\cO(2)$ denotes terms that are quadratic or higher order in $\delta_1$, $\delta_2$, and $\delta_3$.
By substituting these into \eqref{eq:Abproof1}
and Taylor expanding $\phi$ and $\dot{\phi}$ to first order, we obtain
\begin{equation}
\begin{split}
a_{11} \delta_1 + a_{12} \delta_2 + b_1 \delta_3
&= \dot{\phi} \,\frac{\left( a_{21} \delta_1 + a_{22} \delta_2 + b_2 \delta_3 \right)}{\omF}
+ \frac{\partial \phi}{\partial x} \,\delta_1
+ \frac{\partial \phi}{\partial t_0} \,\frac{\delta_2}{\omF}
+ \frac{\partial \phi}{\partial \cA} \,\delta_3 + \cO(2), \\
0 &= \ddot{\phi} \,\frac{\left( a_{21} \delta_1 + a_{22} \delta_2 + b_2 \delta_3 \right)}{\omF}
+ \frac{\partial \dot{\phi}}{\partial x} \,\delta_1
+ \frac{\partial \dot{\phi}}{\partial t_0} \,\frac{\delta_2}{\omF}
+ \frac{\partial \dot{\phi}}{\partial \cA} \,\delta_3 + \cO(2),
\end{split}
\label{eq:Abproof2}
\end{equation}
where each quantity involving $\phi$ is evaluated at grazing.
Matching terms in \eqref{eq:Abproof2} gives
\begin{equation}
\begin{aligned}
a_{11} &= \frac{\partial \phi}{\partial x} - \frac{\dot{\phi}}{\ddot{\phi}} \,\frac{\partial \dot{\phi}}{\partial x}, \qquad \qquad &
a_{12} &= \frac{1}{\omF} \left( \frac{\partial \phi}{\partial t_0}
- \frac{\dot{\phi}}{\ddot{\phi}} \,\frac{\partial \dot{\phi}}{\partial t_0} \right), \\
a_{21} &= -\frac{\omF}{\ddot{\phi}} \,\frac{\partial \dot{\phi}}{\partial x}, \qquad \qquad &
a_{22} &= -\frac{1}{\ddot{\phi}} \,\frac{\partial \dot{\phi}}{\partial t_0}, \\
b_1 &= \frac{\partial \phi}{\partial \cA} - \frac{\dot{\phi}}{\ddot{\phi}} \,\frac{\partial \dot{\phi}}{\partial \cA}, \qquad \qquad &
b_2 &= -\frac{\omF}{\ddot{\phi}} \,\frac{\partial \dot{\phi}}{\partial \cA},
\end{aligned}
\label{eq:Abproof3}
\end{equation}
where again each quantity involving $\phi$ is evaluated at grazing.
By differentiating the explicit expressions \eqref{eq:flow} and \eqref{eq:phip},
and evaluating these at grazing, we obtain
\begin{equation}
\begin{aligned}
\dot{\phi} &= 0, &
\ddot{\phi} &= -\omF^2, \\
\frac{\partial \phi}{\partial x} &= \re^{\frac{-2 \pi \zeta}{\omF}}
\left( \cos \left( \tfrac{2 \pi \omDN}{\omF} \right) + \tfrac{\zeta}{\omDN} \,\sin \left( \tfrac{2 \pi \omDN}{\omF} \right) \right), &
\frac{\partial \dot{\phi}}{\partial x} &= -\tfrac{1}{\omDN} \,\re^{\frac{-2 \pi \zeta}{\omF}} \sin \left( \tfrac{2 \pi \omDN}{\omF} \right), \\
\frac{\partial \phi}{\partial t_0} &= \tfrac{\omF^2}{\omDN} \,\re^{\frac{-2 \pi \zeta}{\omF}} \sin \left( \tfrac{2 \pi \omDN}{\omF} \right), &
\frac{\partial \dot{\phi}}{\partial t_0} &= \omF^2 \re^{\frac{-2 \pi \zeta}{\omF}}
\left( \cos \left( \tfrac{2 \pi \omDN}{\omF} \right) - \tfrac{\zeta}{\omDN} \,\sin \left( \tfrac{2 \pi \omDN}{\omF} \right) \right), \\
\frac{\partial \phi}{\partial \cA} &= \frac{1 - \re^{\frac{-2 \pi \zeta}{\omF}}
\left( \cos \left( \tfrac{2 \pi \omDN}{\omF} \right) + \tfrac{\zeta}{\omDN} \,\sin \left( \tfrac{2 \pi \omDN}{\omF} \right) \right)}
{\cA_{\rm graz}(\omF)}, &
\frac{\partial \dot{\phi}}{\partial \cA} &= \frac{\re^{\frac{-2 \pi \zeta}{\omF}} \sin \left( \tfrac{2 \pi \omDN}{\omF} \right)}
{\omDN \cA_{\rm graz}(\omF)}.
\end{aligned}
\label{eq:Abproof4}
\end{equation}
By substituting these into \eqref{eq:Abproof3}, we obtain
\begin{align}
a_{11} &= \re^{-\frac{2 \pi \zeta}{\omF}} \left( \cos \left( \tfrac{2 \pi \omDN}{\omF} \right)
+ \tfrac{\zeta}{\omDN} \,\sin \left( \tfrac{2 \pi \omDN}{\omF} \right) \right), &
a_{12} &= \tfrac{\omF}{\omDN} \re^{-\frac{2 \pi \zeta}{\omF}} \sin \left( \tfrac{2 \pi \omDN}{\omF} \right), \nonumber \\
a_{21} &= -\tfrac{1}{\omDN \omF} \re^{-\frac{2 \pi \zeta}{\omF}} \sin \left( \tfrac{2 \pi \omDN}{\omF} \right), &
a_{22} &= \re^{-\frac{2 \pi \zeta}{\omF}} \left( \cos \left( \tfrac{2 \pi \omDN}{\omF} \right)
- \tfrac{\zeta}{\omDN} \,\sin \left( \tfrac{2 \pi \omDN}{\omF} \right) \right), \nonumber \\
b_1 &= \frac{1 - \re^{-\frac{2 \pi \zeta}{\omF}} \left( \cos \left( \tfrac{2 \pi \omDN}{\omF} \right)
+ \tfrac{\zeta}{\omDN} \,\sin \left( \tfrac{2 \pi \omDN}{\omF} \right) \right)}{\cA_{\rm graz}(\omF)}, &
b_2 &= \frac{\re^{-\frac{2 \pi \zeta}{\omF}} \sin \left( \tfrac{2 \pi \omDN}{\omF} \right)}{\omDN \omF \cA_{\rm graz}(\omF)}, \nonumber
\end{align}
which are the desired formulas \eqref{eq:oscA} and \eqref{eq:oscb}.
\end{proof}

\begin{proof}[Proof of Proposition \ref{pr:imp1}]
By extending the asymptotic calculations of the previous proof
from first-order to second-order, we obtain
\begin{equation}
\xi_1 = \frac{1}{\omF^2} \left( a_{22}^2 \left( \ddot{\phi} - \frac{\dot{\phi} \dddot{\phi}}{\ddot{\phi}} \right)
+ 2 a_{22} \left( \frac{\partial \dot{\phi}}{\partial t_0} - \frac{\dot{\phi}}{\ddot{\phi}} \,\frac{\partial \ddot{\phi}}{\partial t_0} \right)
+ \frac{\partial^2 \phi}{\partial t_0^2} - \frac{\dot{\phi}}{\ddot{\phi}} \,\frac{\partial^2 \dot{\phi}}{\partial t_0^2} \right).
\nonumber
\end{equation}
At grazing
\begin{equation}
\frac{\partial^2 \phi}{\partial t_0^2} = -\omF^2 \re^{\frac{-2 \pi \zeta}{\omF}}
\left( \cos \left( \tfrac{2 \pi \omDN}{\omF} \right) - \tfrac{3 \zeta}{\omDN} \,\sin \left( \tfrac{2 \pi \omDN}{\omF} \right) \right),
\label{eq:d2phidt02}
\end{equation}
and by also using the formulas \eqref{eq:Abproof4}, we obtain
\begin{equation}
\xi_1 = \re^{\frac{-4 \pi \zeta}{\omF}} \left( \cos \left( \tfrac{2 \pi \omDN}{\omF} \right) - \tfrac{\zeta}{\omDN} \,\sin \left( \tfrac{2 \pi \omDN}{\omF} \right) \right)^2
- \re^{\frac{-2 \pi \zeta}{\omF}} \left( \cos \left( \tfrac{2 \pi \omDN}{\omF} \right) - \tfrac{3 \zeta}{\omDN} \,\sin \left( \tfrac{2 \pi \omDN}{\omF} \right) \right).
\label{eq:xi1proof2}
\end{equation}
At resonance, $\sin \left( \frac{2 \pi \omDN}{\omF} \right) = \sin(n \pi) = 0$,
and $\cos \left( \frac{2 \pi \omDN}{\omF} \right) = \cos(n \pi) = (-1)^n$,
so \eqref{eq:xi1proof2} reduces to
\begin{equation}
\xi_1 = E_1 \left( E_1 - (-1)^n \right),
\label{eq:xi1proof3}
\end{equation}
where $E_1 = \re^{\frac{-2 \pi \zeta}{\omF}}$.
Also $a_{11} = a_{22} = (-1)^n E_1$, so by evaluating \eqref{eq:c1} we obtain
\begin{align}
s^\pm_1 &= \left( 1 \mp (-1)^n E_1 \right) \left( (-1)^n \epsilon^2 E_1 \mp 1 \right)
+ \frac{(1+\epsilon)^2 E_1 \left( E_1 - (-1)^n \right)}{1 - (-1)^n E_1} \nonumber \\
&= \mp \left( 1 \pm 2 (-1)^n \epsilon E_1 + \epsilon^2 E_1^2 \right) \,,
\label{eq:xi1proof4}
\end{align}
using also \eqref{eq:oscphipsigamma} and \eqref{eq:oscalpha}.
This verifies the formulas for $s^\pm_1$ in \eqref{eq:impFormula1odd} and \eqref{eq:impFormula1even}.

For the impact oscillator we are using $\eta = \omF - \omF^*$,
so $a_{12}' = \frac{\partial a_{12}}{\partial \omF}$, evaluated at grazing.
By differentiating
$a_{12} = \frac{\omF}{\omDN} \re^{-\frac{2 \pi \zeta}{\omF}} \sin \left( \tfrac{2 \pi \omDN}{\omF} \right)$
with respect to $\omF$, we obtain
\begin{equation}
a_{12}' = -\frac{2 \pi (-1)^n E_1}{\omF},
\label{eq:xi1proof5}
\end{equation}
using again
$\sin \left( \frac{2 \pi \omDN}{\omF} \right) = 0$ and $\cos \left( \frac{2 \pi \omDN}{\omF} \right) = (-1)^n$.
Finally, by using \eqref{eq:xi1proof4} and \eqref{eq:xi1proof5}
in \eqref{eq:coeffs1}, we arrive at the formulas for
$c_{{\rm SN},1}$ and $c_{{\rm PD},1}$ given in \eqref{eq:impFormula1odd} and \eqref{eq:impFormula1even}.
\end{proof}

\begin{proof}[Proof of Proposition \ref{pr:impp}]
With $p \ge 2$, the calculation of $\xi_p$ is identical to that for $p = 1$,
given above in the proof of Proposition \ref{pr:imp1},
except corresponds to an evolution time of $\frac{z' - z + 2 \pi p}{\omF}$ instead of $\frac{z' - z + 2 \pi}{\omF}$.
It follows that, instead of \eqref{eq:xi1proof2}, we have
\begin{equation}
\xi_p = \re^{\frac{-4 \pi p \zeta}{\omF}} \left( \cos \left( \tfrac{2 \pi p \omDN}{\omF} \right) - \tfrac{\zeta}{\omDN} \,\sin \left( \tfrac{2 \pi p \omDN}{\omF} \right) \right)^2
- \re^{\frac{-2 \pi p \zeta}{\omF}} \left( \cos \left( \tfrac{2 \pi p \omDN}{\omF} \right) - \tfrac{3 \zeta}{\omDN} \,\sin \left( \tfrac{2 \pi p \omDN}{\omF} \right) \right).
\label{eq:xipproof2}
\end{equation}
At resonance, $\sin \left( \frac{2 \pi p \omDN}{\omF} \right) = \sin(\pi) = 0$
and $\cos \left( \frac{2 \pi p \omDN}{\omF} \right) = \cos(\pi) = -1$,
so \eqref{eq:xipproof2} reduces to
\begin{equation}
\xi_p = E_p \left( E_p + 1 \right),
\label{eq:xipproof3}
\end{equation}
where $E_p = \re^{\frac{-2 \pi p \zeta}{\omF}}$.
Notice $E_p = \delta^{\frac{p}{2}}$, by \eqref{eq:osctaudelta},
so by evaluating \eqref{eq:cp} we obtain
\begin{align}
s^\pm_p &= \left( 1 \pm E_p \right) \left( -\epsilon^2 E_p \mp 1 \right)
+ (1+\epsilon)^2 E_p \nonumber \\
&= \mp \left( 1 \mp \epsilon E_p \right)^2 \,,
\label{eq:xipproof4}
\end{align}
using also \eqref{eq:oscphipsigamma} and \eqref{eq:oscalpha}.

Next, $\kappa_p' = \frac{\partial \kappa_p}{\partial \omF}$,
where $\kappa_p = \tau - 2 \sqrt{\delta} \cos \left( \frac{\pi}{p} \right)$.
So by using the formulas \eqref{eq:osctaudelta} for $\tau$ and $\delta$, we obtain
\begin{equation}
\kappa_p' = \tfrac{4 \pi}{\omF^2} \,\re^{-\frac{2 \pi \zeta}{\omF}}
\left( \zeta \cos \left( \tfrac{2 \pi \omDN}{\omF} \right) + \omDN \sin \left( \tfrac{2 \pi \omDN}{\omF} \right) - \zeta \cos \left( \tfrac{\pi}{p} \right) \right).
\nonumber
\end{equation}
At resonance, $\cos \left( \frac{2 \pi \omDN}{\omF} \right) = \cos \left( \frac{\pi}{p} \right)$
and $\sin \left( \frac{2 \pi \omDN}{\omF} \right) = \sin \left( \frac{\pi}{p} \right)$, so
\begin{equation}
\kappa_p' = \tfrac{4 \pi \omDN}{\omF^2} \,\re^{-\frac{2 \pi \zeta}{\omF}} \sin \left( \tfrac{\pi}{p} \right).
\label{eq:xipproof5}
\end{equation}
Finally, by substituting \eqref{eq:xipproof4} and \eqref{eq:xipproof5}
into \eqref{eq:coeffsp}, and further using the formulas listed in \S\ref{sub:oscGrazParams},
we obtain after simplification $c_{{\rm SN},p}$ and $c_{{\rm PD},p}$ as given in \eqref{eq:impFormulap}.
\end{proof}

{\footnotesize
\bibliographystyle{unsrt}
\bibliography{ResonantGrazingArXivBib}
}

\end{document}